\documentclass[a4paper,12pt]{amsart}

\usepackage{fullpage}
\usepackage{amsmath,amssymb,graphicx,psfrag}
\usepackage[T1]{fontenc}
\usepackage{amsmath,amssymb,ifthen,moreverb}
\usepackage{color}
\usepackage{subfigure}

\definecolor{darkgreen}{rgb}{0,0.4,0}

\definecolor{darkblue}{rgb}{0,0,.8}
\definecolor{darkred}{rgb}{0.8,0,0}
\usepackage{hyperref}
\hypersetup{
  pdftex=true,
  colorlinks=true,
  allcolors = black,
  linkcolor=black, 
  menucolor=darkblue,
  urlcolor = darkblue,
  pagecolor=darkblue,
  frenchlinks=false,
  pdfborder={0 0 0},
  naturalnames=false,
  hypertexnames=false,
  breaklinks
}
\usepackage[capitalize,nameinlink]{cleveref}
\crefname{section}{Section}{Sections}
\crefname{subsection}{Section}{Sections}

\Crefname{figure}{Figure}{Figures}

\crefformat{equation}{\textup{#2(#1)#3}}
\crefrangeformat{equation}{\textup{#3(#1)#4--#5(#2)#6}}
\crefmultiformat{equation}{\textup{#2(#1)#3}}{--\textup{#2(#1)#3}}
{, \textup{#2(#1)#3}}{, and \textup{#2(#1)#3}}
\crefrangemultiformat{equation}{\textup{#3(#1)#4--#5(#2)#6}}%
{ and \textup{#3(#1)#4--#5(#2)#6}}{, \textup{#3(#1)#4--#5(#2)#6}}{, and \textup{#3(#1)#4--#5(#2)#6}}

\Crefformat{equation}{#2Equation~\textup{(#1)}#3}
\Crefrangeformat{equation}{Equations~\textup{#3(#1)#4--#5(#2)#6}}
\Crefmultiformat{equation}{Equations~\textup{#2(#1)#3}}{ and \textup{#2(#1)#3}}
{, \textup{#2(#1)#3}}{, and \textup{#2(#1)#3}}
\Crefrangemultiformat{equation}{Equations~\textup{#3(#1)#4--#5(#2)#6}}%
{ and \textup{#3(#1)#4--#5(#2)#6}}{, \textup{#3(#1)#4--#5(#2)#6}}{, and \textup{#3(#1)#4--#5(#2)#6}}

\crefdefaultlabelformat{#2\textup{#1}#3}

\newtheorem{theorem}{Theorem}
\newtheorem{lemma}[theorem]{Lemma}
\newtheorem{corollary}[theorem]{Corollary}
\newtheorem{proposition}[theorem]{Proposition}
\newtheorem{algorithm}[theorem]{Algorithm}
\newtheorem{remark}[theorem]{Remark}

\def\subsection#1{%
 \vskip2mm
 \refstepcounter{subsection}%
 {\bf\arabic{section}.\arabic{subsection}.~#1.~}%
}

\def\A{\mathbf{A}}            
\def\b{\mathbf{b}} 
\def\c{c}

\def\eps{\varepsilon}

\def\set#1#2{\big\{#1\,:\,#2\big\}}

\newcommand{\norm}[3][]{#1\|#2#1\|_{#3}}

\newcommand{\enorm}[2][]{#1|\hspace*{-.3mm}#1|\hspace*{-.3mm}#1|#2#1|\hspace*{-.3mm}#1|\hspace*{-.3mm}#1|}

\def\div{{\rm div}}

\def\osc{ {\rm osc} }

\def\<{\langle\hspace*{-.9mm}\langle}
\def\>{\rangle\hspace*{-.9mm}\rangle}

\def\product#1#2{(#1\hspace*{.5mm},#2)}

\def\N{{\mathbb N}}

\def\R{{\mathbb R}}
\def\T{{\mathbb T}}

\def\nf{{\mathbf n}}

\def\AA{{\mathcal A}}  
\def\CC{{\mathcal C}}

\def\II{{\mathcal I}}
\def\MM{{\mathcal M}}
\def\NN{{\mathcal N}}
\def\OO{{\mathcal O}}
\def\PP{{\mathcal P}}
\def\RR{{\mathcal R}}

\def\SS{{\mathcal S}}
\def\TT{{\mathcal T}}

\def\refine{{\tt refine}}

\def\FFt{\FF_{T}}               

\def\[{[\hspace*{-0.8mm}[}
\def\]{]\hspace*{-0.8mm}]}
\newcommand{\jump}[2][]{%
 \ifthenelse{\equal{#1}{}}
 {[\hspace*{-0.8mm}[#2]\hspace*{-0.8mm}]}%
 {\left[\hspace*{-1.6mm}\left[#2\right]\hspace*{-1.6mm}\right]}%
}

\def\diam{{\rm diam}}

\def\Cmark{C_{\rm mark}}
\def\Clin{C_{\rm lin}}
\def\qlin{q_{\rm lin}}

\def\qest{q_{\rm est}}
\def\Copt{C_{\rm opt}}
\def\Crel{C_{\rm rel}}

\def\Cmns{C_{\rm MNS}}%
\def\Cgal{C_{\rm galerkin}}%
\def\Cbil{C_{\rm bil}}%
\def\Cstab{C_{\rm stab}}%
\def\Cosc{C_{\rm osc}}%

\def\Cell{C_{\rm ell}}%
\def\CAbil{C_{\rm \AA}}%
\def\CAbil{C_{\rm cont}}%
\def\Cdual{C_{\rm aux}}

\def\refine{{\tt refine}}

\def\FF{{\mathcal F}}

\def\facet{F}
\def\star{\times}

\title[Adaptive FVM for general PDEs]{Adaptive vertex-centered finite volume methods\\ for general second-order linear elliptic PDEs}
\author{Christoph Erath}
\address{TU Darmstadt, Department of Mathematics, Dolivostra\ss{}e 15, 64293 Darmstadt, Germany}
\email{Erath@mathematik.tu-darmstadt.de\quad\rm(corresponding author)}

\author{Dirk Praetorius}
\address{TU Wien, Institute for Analysis and Scientific Computing,
Wiedner Hauptstra\ss{}e 8-10, 1040 Wien, Austria}
\email{Dirk.Praetorius@asc.tuwien.ac.at}

\thanks{C. Erath (corresponding author): TU Darmstadt, Germany; Erath@mathematik.tu-darmstadt.de}
\thanks{D. Praetorius: TU Wien, Austria; Dirk.Praetorius@asc.tuwien.ac.at}
\thanks{The second author acknowledges support through the research project
\emph{Optimal adaptivity for BEM and FEM-BEM coupling} funded by the Austrian Science Fund (FWF) under
grant P27005 and of the special research program
\emph{Taming complexity in partial differential systems} funded by the Austrian Science Fund (FWF) under grant F65}

\date{\bf\color{red}\today}
\begin{document}

\begin{abstract}
We prove optimal convergence rates for the discretization of a general second-order linear elliptic PDE with an adaptive 
vertex-centered finite volume scheme. While our prior work
Erath and Praetorius [SIAM J. Numer. Anal., 54 (2016), pp. 2228--2255] was restricted to symmetric
problems, 
the present analysis also covers non-symmetric problems and hence the important case of present convection.
\end{abstract}



\maketitle

\section{Introduction}
%
In this work we consider a general second-order linear elliptic PDE and approximate the solution
with an adaptive vertex-centered finite volume method (FVM). Finite volume methods are well established
in fluid mechanics, since they naturally preserve numerical flux conservation.

\subsection{Model problem}
\label{section:modelproblem}%
Let $\Omega\subset \R^d$, $d=2,3$, be a bounded Lipschitz domain with 
polygonal boundary $\Gamma:=\partial\Omega$.
As model problem, we consider the following stationary diffusion problem: Given $f\in L^2(\Omega)$, find $u\in H^1(\Omega)$ such that
\begin{align}
\label{eq:model}
 \div (-\A \nabla u+\b u)+\c u  &= f \quad \text{in }\Omega\qquad\text{and}\qquad
                               u = 0\quad \text{on }\Gamma.
\end{align}
We suppose that the diffusion matrix $\A=\A(x)\in\R^{d\times d}$ is bounded, symmetric, and uniformly positive definite, i.e., there exist constants $\lambda_{\rm min},\lambda_{\rm max}>0$ such that
\begin{align}\label{eq:A}
 \lambda_{\rm min}\,|\mathbf{v}|^2\leq \mathbf{v}^T\A(x)\mathbf{v}\leq \lambda_{\rm max}\,|\mathbf{v}|^2
 \quad\text{for all } \mathbf{v}\in \R^d \text{ and almost all }x\in\Omega.
\end{align}
Let $\TT_0$ be a given initial triangulation of $\Omega$; see \cref{section:triangulations} below.
For convergence of FVM and well-posedness of the residual error estimator, 
we additionally require that $\A(x)$ is piecewise Lipschitz continuous, i.e., 
\begin{align}\label{eq3:A}
 \A\in W^{1,\infty}(T)^{d\times d}
 \quad\text{for all } T\in\TT_0.
\end{align} 
We suppose that the lower-order terms satisfy 
the assumption 
\begin{align}
  \label{eq:bcestimate1}
  \b\in W^{1,\infty}(\Omega)^d
  \text{ and }
  \c\in L^{\infty}(\Omega)
  \text{ with }  
  \frac{1}{2}\div\,\b+\c \geq 0
  \text{ almost everywhere on }\Omega.
\end{align}
With $\product{\phi}{\psi}_\Omega = \int_\Omega \phi(x)\psi(x)\,dx$ being the $L^2$-scalar product,
the weak formulation of the model problem~\cref{eq:model} reads as follows: Find $u\in H^1_0(\Omega)$ such that 
\begin{align}
\label{eq:weakform}
 \AA(u,w):= \product{\A\nabla u-\b u}{\nabla w}_{\Omega}
 +\product{\c u}{w}_{\Omega} = \product{f}{w}_{\Omega}
 \quad\text{for all } w\in H^1_0(\Omega).
\end{align}
According to our assumptions~\cref{eq:A,eq:bcestimate1}, 
the bilinear form $\AA(\cdot,\cdot)$ is continuous and elliptic on $H^1_0(\Omega)$. 
Existence and uniqueness of the 
solution $u\in H^1_0(\Omega)$ of~\cref{eq:weakform} thus follow from the Lax-Milgram theorem.
Moreover, the operator
induced quasi norm $\enorm{\cdot}$ satisfies that
\begin{align}\label{eq:enorm:equivalent}
 \Cell\norm{v}{H^1(\Omega)}^2
 \leq \enorm{v}^2 := \AA(v,v)\leq 
 \CAbil\norm{v}{H^1(\Omega)}^2
 \qquad\text{for all }v\in H_0^1(\Omega),
\end{align}
where $\Cell>0$ depends only on $\lambda_{\rm min}$ and $\Omega$, whereas $\CAbil>0$ depends only on
$\lambda_{\rm max}$, $\norm{\b}{L^\infty(\Omega)}$ and $\norm{\c}{L^\infty(\Omega)}$.

\subsection{Adaptive FVM}
%
In the past 20 years, there have been major contributions to the mathematical understanding of  adaptive mesh refinement 
algorithms, mainly in the context of the finite element method (FEM).
While the seminal works~\cite{doerfler96,mns00,bdd04,stevenson07,ckns} were restricted to symmetric operators, the recent works~\cite{Mekchay:2005-1,Cascon:2012-1,Feischl:2014,bhp17} proved
convergence of adaptive FEM with optimal algebraic rates for general second order linear elliptic PDEs.
The work~\cite{axioms} gives an exhausted overview of the developments
and it gains, in an abstract framework, a general recipe to prove optimal adaptive convergence rates
of adaptive mesh refining algorithms.
Basically, the numerical discretization scheme, the {\sl a~posteriori} error estimator and the adaptive algorithm have to fulfill four criteria (called {\em axioms} in~\cite{axioms}), namely,
\emph{stability on non-refined elements}, \emph{reduction on refined elements}, \emph{general quasi-orthogonality}, and
\emph{discrete reliability}.
Building upon these findings,
our recent work~\cite{Erath:2016-1} gave the first proof of convergence of adaptive FVM with optimal algebraic rates for a symmetric model problem~\cref{eq:model} with $\b=0$ and $\c=0$. 

\subsection{Contributions and outline}
In this work, we are in particular interested in the non-symmetric model problem with
$\b\not = 0$ in \cref{eq:model}.
The proofs of \emph{stability on non-refined elements}, \emph{reduction on refined elements}, and
\emph{discrete reliability} follow basically the proofs in~\cite{Erath:2016-1}; see \cref{subsec:linearconv}. 
Thus, the major contribution of the present work is the proof of the \emph{general quasi-orthogonality} property for the non-symmetric problem, which is satisfied under some mild regularity assumptions on the dual problem.
Similar assumptions are required in~\cite{Mekchay:2005-1,Cascon:2012-1} 
to prove convergence for an adaptive FEM procedure. Moreover, we note that~\cite{Mekchay:2005-1,Cascon:2012-1} require slightly more restrictions on the model data (namely, $\div(\b)=0$) and on the mesh-refinement (the so-called \emph{interior node property}) 
for proving quasi-orthogonality which are avoided in the present analysis. 

At this point, we note that~\cite{Feischl:2014,bhp17} improve the FEM result of~\cite{Mekchay:2005-1,Cascon:2012-1} by a different approach. Instead of the duality argument,  the analysis exploits the {\sl a~priori} convergence of FEM solutions (which follows from the classical C\'ea lemma) by splitting the operator into a symmetric and elliptic part and a compact perturbation.
In particular, there is no duality argument applied and, therefore, no additional 
regularity assumption is required. However, it seems to be difficult to transfer the 
analysis of~\cite{Feischl:2014,bhp17} to FVM due to the lack of the C\'ea lemma.

We also mention that unlike the FEM literature, a direct proof of the \emph{general quasi-orthogonality} is not available for FVM due to the lack of Galerkin orthogonality.
Instead, the FVM work~\cite{Erath:2016-1} first proves linear convergence
which relies on a quasi-Galerkin orthogonality \cite[Lemma~11]{Erath:2016-1} for FVM. 
Unfortunately, this auxiliary result does not hold for non-symmetric problems.

Hence, to handle the non-symmetric case, the missing Galerkin orthogonality 
and the lack of an optimal $L^2$ estimate for FVM seem to be the bottlenecks.
To overcome these difficulties, we first estimate the FVM error in the bilinear form by oscillations in \cref{lem:galerkinosc}.  
Then we provide a new $L^2$-type estimate in \cref{lemma:dual} which depends
on the regularity of the corresponding dual problem plus oscillations.
These two results provide the key arguments to prove a quasi-Galerkin orthogonality in~\cref{lemma:orth}. Unlike the literature,
this estimate also includes a mesh-size weighted estimator term.  
With the aid of the previous results, we 
show linear convergence in \cref{theorem:mns:linear}, where the proof relies on the previous results.
Finally, optimal algebraic convergence rates are guaranteed by \cref{theorem:mns} which follows directly from the literature.

We remark that the proposed \cref{algorithm:mns} additionally marks
oscillations to overcome the lack of classical Galerkin orthogonality.
Note that this is not required for adaptive FEM. However, since FVM is not a best approximation method,
the proposed approach appears
to be rather natural. In practice, however, this additional marking is negligible;
see also~\cite[Remark 12]{Erath:2016-1}.
Overall, the present work seems to be the first which proves convergence with optimal rates of an adaptive FVM algorithm
for the solution of general second-order linear elliptic PDEs.

\section{Preliminaries}
This section introduces the notation, the discrete scheme, as well as the 
residual {\sl a~posteriori} error estimator. In particular, 
we fix our notation used throughout this work.

\subsection{General notation}
Throughout, $\nf$ denotes the unit normal vector to the boundary pointing outward the respective domain.
In the following, we mark the mesh dependency of quantities by appropriate indices, e.g.,
$u_\ell$ is the solution on the triangulation $\TT_\ell$.
Furthermore, $\lesssim$ abbreviates $\leq$ up to some (generic) multiplicative constant which is clear from the context.

\subsection{Triangulations}
\label{section:triangulations}%
The FVM relies on two partitions of $\Omega$, the \emph{primal mesh} $\TT_\star$ and the associated \emph{dual mesh} $\TT_\star^*$. 
The primal mesh $\TT_\star$ is a regular triangulation of $\Omega$ into non-degenerate closed 
triangles/tetrahedra $T\in\TT_\star$, 
where the possible discontinuities of the coefficient matrix $\A$ are aligned with $\TT_\star$. Define the local mesh-size function 
\begin{align}
 h_\star\in L^\infty(\Omega),
 \quad
 h_\star|_T:=h_T := |T|^{1/d}
 \quad\text{for all }T\in\TT_\star.
\end{align} 
Let $\diam(T)$ be the Euclidean diameter of $T$. Suppose that $\TT_\star$ is $\sigma$-shape regular, i.e.,
\begin{align}
 \label{eq:sigma}
 \max_{T\in\TT_\star}\frac{\diam(T)}{|T|^{1/d}} \le \sigma < \infty.
\end{align}
Note that this implies $h_T \le \diam(T) \le\sigma\,h_T$.
Let $\NN_\star$ (or $\NN_\star^\Omega$) denote the set of all (or all interior) nodes.
Let $\FF_\star$ (or $\FF_\star^\Omega$) denote the set of all (or all interior) facets.
For $T\in\TT_\star$, let $\FF_T := \set{F\in\FF_\star}{F\subseteq\partial T}$ be the set of facets of $T$. Moreover, 
\begin{align}
 \omega_\star(T):=\bigcup\set{T'\in\TT_\star}{T\cap T'\neq\emptyset}\subseteq\overline{\Omega}
\end{align}
denotes the element patch of $T$ in $\TT_\star$.

\begin{figure}
\begin{center}
	\subfigure[\label{subfig:dualmesh} $\TT_\star$ (triangles) and $\TT_\star^*$ (grey boxes).]
	{\includegraphics[width=0.25\textwidth]
	{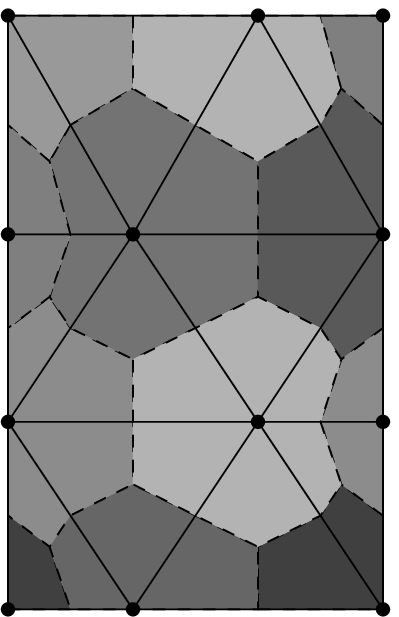}}
	\hspace{0.1\textwidth}
	\subfigure[\label{subfig:nvb}Refinement of a triangle by newest vertex bisection (NVB).]
	{\includegraphics[width=0.45\textwidth]
	{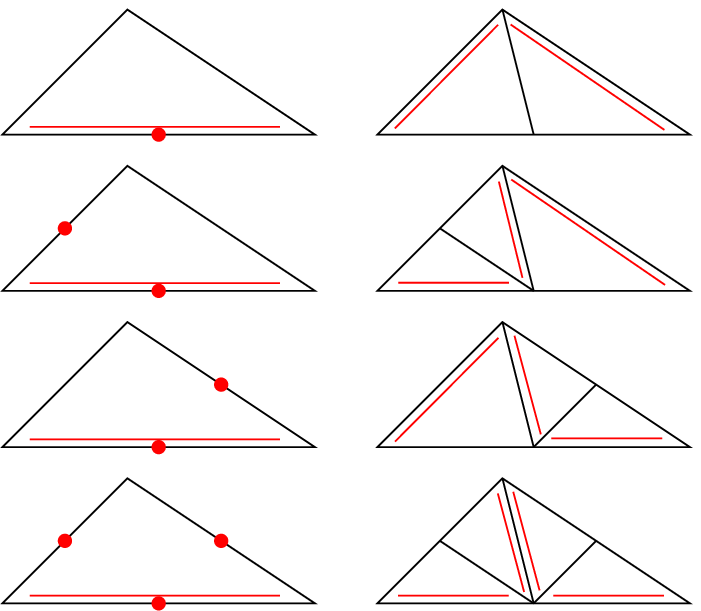}}
\end{center}
\caption{\label{fig:mesh}%
Construction of the dual mesh $\TT_\star^*$ from the primal mesh $\TT_\star$ in 2D (left) and 2D newest vertex bisection (right). Each triangle has a reference edge (indicated by the double line). If edges are marked for refinement (indicated by dots), the resulting configurations are shown.}
\end{figure}

The associated \emph{dual mesh} $\TT_\star^*$ is obtained as follows:
For $d=2$, connect the center of gravity of an element $T\in\TT_\star$ with the midpoint of an edge of $\partial T$. 
These lines define the non-degenerate closed polygons $V_i\in\TT_\star^*$; see \cref{subfig:dualmesh}.
For $d=3$,
we first connect the center of gravity of $T\in\TT_\star$ with each center of gravity of the four
faces of $F\in\FF_T$ by straight lines. 
Then, as in the 2D case, we connect each center of gravity of $F\in\FF_T$ to
the midpoints of the edges of the face $F$. Note that this forms polyhedrons $V_i\in\TT_\star^*$.
In 2D and 3D, each volume $V_i\in\TT_\star^*$
is uniquely associated with a node $a_i$ of $\TT_\star$.

\subsection{Discrete spaces}
For a partition $\MM$ of $\Omega$ and $p\in\N_0$, let
\begin{align}
 \PP^p(\MM) := \set{v:\Omega\to\R}{\forall M\in\MM\quad v|_M\text{ is polynomial of degree }\le p}
\end{align}
be the space of $\MM$-piecewise polynomials of degree $p$.
With this at hand, let
\begin{align}
 \SS^1(\TT_\star) := \PP^1(\TT_\star)\cap H^1(\Omega) = \set{v_\star\in C(\Omega)}{\forall T\in\TT_\star\quad v_\star|_T\text{ is affine}}.
\end{align}
Then the discrete ansatz space 
\begin{align}
 \SS^1_0(\TT_\star):= \SS^1(\TT_\star)\cap H^1_0(\Omega) 
 = \set{v_\star\in\SS^1(\TT_\star)}{v_\star|_\Gamma = 0}
\end{align}
consists of all $\TT_\star$-piecewise affine and globally continuous functions which are zero on $\Gamma$. 
By convention, the discrete test space
\begin{align}
 \PP^0_0(\TT^*_\star):= \set{v_\star^*\in\PP^0(\TT_\star^*)}{v_\star^*|_\Gamma=0}
\end{align}%
consists all $\TT_\star^*$-piecewise constant functions which are zero on all $V\in \TT^*_\star$ 
with $\partial V\cap \Gamma\not = \emptyset$.

\subsection{Mesh-refinements}
For local mesh-refinement, we employ newest vertex bisection (NVB); 
see~\cite{stevenson08,kpp13} and \cref{subfig:nvb}.
Below, we use the following notation:
First, $\TT':=\refine(\TT,\MM)$ denotes the coarsest conforming triangulation 
generated by NVB from a conforming triangulation $\TT$
such that all marked elements $\MM\subseteq \TT$ have been refined, i.e., $\MM \subseteq \TT \backslash \TT'$.
Second, we simply write $\TT' \in \refine(\TT)$, if $\TT'$ is an 
arbitrary refinement of $\TT$, i.e., there exists a finite number of refinements steps $j=1,\ldots,n$ such
that $\TT'=\TT'_n$ can be generated from $\TT=\TT'_0$ with marked elements $\MM'_j\subseteq \TT'_{j}$ and
$\TT'_j=\refine(\TT'_{j-1},\MM'_{j-1})$.
Note that NVB guarantees that there exist only finitely many shapes of triangles and patches in $\TT' \in \refine(\TT)$. These shapes are determined by $\TT$. In particular, the meshes $\TT' \in \refine(\TT)$ are uniformly $\sigma$-shape regular~\cref{eq:sigma}, where $\sigma$ depends only on $\TT$.

\subsection{Vertex-centered finite volume method (FVM)}
The FVM approximates the solution $u\in H^1_0(\Omega)$ of~\cref{eq:weakform} by some  $u_\star\in\SS^1_0(\TT_\star)$. The scheme is based on the balance equation
over $\TT_\star^*$ and reads in variational form as follows:
Find $u_\star\in\SS^1_0(\TT_\star)$ such that
\begin{align}
\label{eq:disc_systemfvem}
  \AA_\star(u_\star,w_\star^*)
  &= \product{f}{w_\star^*}_\Omega=\sum_{a_i\in\NN_\star^\Omega}w_\star^*|_{V_i}
  \int_{V_i} f\,dx
  \quad \text{for all } w_\star^*\in\PP^0_0(\TT_\star^*).
\end{align}
For all $v_\star\in\SS^1_0(\TT_\star)$
and all $w_\star^*\in\PP^0_0(\TT_\star^*)$,
the bilinear form reads 
\begin{align*}
  \AA_\star(v_\star,w_\star^*):=\sum_{a_i\in\NN_\star^\Omega}w_\star^*|_{V_i}
  \Big(\int_{\partial V_i}(-\A \nabla v_\star+\b v_\star)\cdot \nf\,ds
  +\int_{V_i}\c v_\star\,dx\Big).
\end{align*}
To recall that the FVM is well-posed on sufficiently fine triangulations $\TT_\star$, 
we require the following interpolation operator; see, 
e.g.,~\cite{Erath:2012-1,Erath:2016-1}.

\begin{lemma}
  \label{lem:proplindual}
  With $\chi_i^*\in\PP^0(\TT_\star^*)$ being the characteristic function of $V_i\in\TT_\star^*$, define 
 \begin{align*}
  \II_\star^*:\CC(\overline\Omega)\to\PP^0(\TT^*_{\star}),\quad
  \II_\star^*v:=\sum_{a_i\in\NN_\star}v(a_i)\chi_i^*.
\end{align*}
Then, for all $T\in\TT_\star$, $\facet\in\FFt$, and $v_\star\in\SS^1(\TT_\star)$, it holds that
  \begin{align}
    \label{eq1:proplindual}
    &\int_T (v_\star-\II_\star^*v_\star)\,dx=0=\int_{\facet}(v_\star-\II_\star^*v_\star)\,ds,\\
    \label{eq2:proplindual}
    &\norm{v_\star-\II_\star^*v_\star}{L^2(T)}\leq h_T \norm{\nabla v_\star}{L^2(T)},\\
    \label{eq3:proplindual}
    &\norm{v_\star-\II_\star^*v_\star}{L^2(\facet)}\leq C h_T^{1/2}\norm{\nabla v_\star}{L^2(T)}.
  \end{align}
In particular, it holds that $\II_\star^*v_\star \in \PP^0_0(\TT_\star^*)$ for all $v_\star\in\SS^1_0(\TT_\star)$.
The constant $C>0$ depends only on $\sigma$-shape regularity of $\TT_\star$.
\qed
\end{lemma}

The following lemma is a key observation for the FVM analysis. 
For Lipschitz continuous $\A$, the proof is 
found in~\cite{Ewing:2002-1,Erath:2012-1}. We 
note that the result transfers directly to the present situation~\cite{Erath:2016-1,Erath:2017-1}, where $\A$ 
satisfies~\cref{eq:A}--\cref{eq3:A} and $\b\not =0$ and $\c\not = 0$.  

\begin{lemma}
 \label{lem:bilinearcomp}
 There exists $\Cbil>0$ such that for all $v_\star,w_\star\in\SS^1_0(\TT_\star)$
 \begin{align}
  \label{eq:bilinearcomp}
  |\AA(v_\star,w_\star) - \AA_\star(v_\star,\II_\star^*w_\star)|
    \leq \Cbil \sum_{T\in\TT_\star} h_T\,\norm{v_\star}{H^1(T)}\norm{w_\star}{H^1(T)}.
 \end{align} 
 Moreover, let $\TT_\star$ be sufficiently fine such that
 $\Cell-\Cbil\norm{h_\star}{L^\infty(\Omega)}>0$, where $\Cell>0$ is the ellipticity constant from~\cref{eq:enorm:equivalent}. 
 Then there exists $\Cstab>0$ such that
 \begin{align}\label{eq:stable}
  \AA_\star(v_\star,\II_\star^*v_\star)
  \ge \Cstab\,\norm{v_\star}{H^1(\Omega)}^2
  \quad\text{for all }v_\star\in\SS^1_0(\TT_\star).
 \end{align}
 In particular, the FVM system~\cref{eq:disc_systemfvem} admits a unique solution $u_\star\in\SS^1_0(\TT_\star)$.
 The constants $\Cbil$ and $\Cstab$ depend only on 
 the data assumptions~\cref{eq:A,eq:bcestimate1}, 
 the $\sigma$-shape regularity of $\TT_\star$, and $\Omega$.
 \qed
\end{lemma}

\subsection{Weighted-residual \textsl{a~posteriori} error estimator}
For all $v_\star\in\SS^1_0(\TT_\star)$, we define the volume residual $R_\star$ and the normal jump $J_\star$  by
\begin{alignat}{2}
  \label{eq:residual}
R_\star(v_\star)|_T&:=f-\div_{\star}(-\A\nabla v_{\star}+\b v_{\star})-\c v_{\star}
\qquad&&\text{for all }T\in\TT_\star,\\ 
\label{eq:jump}
J_\star(v_\star)|_F&:=\jump{\A\nabla v_{\star}}_F 
\qquad&&\text{for all } F\in\FF_\star^\Omega.
\end{alignat}

Here, $\div_\star$ denotes the $\TT_\star$-piecewise divergence operator,
and the normal jump reads
$\jump{\mathbf{g}}|_F:=(\mathbf{g}|_T-\mathbf{g}|_{T'})\cdot\nf$,
where $\mathbf{g}|_T$ denotes the trace of $\mathbf{g}$ from $T$ onto $F$ and 
$\nf$ points from $T$ to $T'$.
Let $\Pi_\times$ be the edgewise or elementwise integral mean operator, i.e.,
\begin{align*}
 (\Pi_\times)|_\tau=\frac{1}{|\tau|}\int_\tau v\,dx \qquad\text{for all }
 \tau\in\TT_\star\cup \FF_\star
 \text{ and all } v\in L^2(\tau).
\end{align*}
For all $T\in\TT_{\star}$, we define the local error indicators and oscillations by
\begin{align}
 \label{eq:etaToscT}
 \begin{split}
 \eta_{\star}(T,v_\star)^2 &:= h_T^2\,\norm{R_\star(v_\star)}{L^2(T)}^2
 + h_T\,\norm{J_\star(v_\star)}{L^2(\partial T\backslash\Gamma)}^2,\\
 \osc_{\star}(T,v_\star)^2 
  &:= h_T^2\,\norm{(1-\Pi_{\star})R_\star(v_\star)}{L^2(T)}^2
    + h_T\,\norm{(1-\Pi_{\star})J_\star(v_\star)}{L^2(\partial T\backslash\Gamma)}^2.
 \end{split}
\end{align}
Then the error estimator $\eta_\star$ and the oscillations $\osc_\star$ are defined by
\begin{align}
 \label{eq:etaosc}
 \eta_\star(v_\star)^2:=\sum_{T\in\TT_\star}\eta_{\star}(T,v_\star)^2 \qquad\text{and}\qquad
 \osc_\star^2(v_\star):=\sum_{T\in\TT_\star}\osc_{\star}(T,v_\star)^2.
\end{align}
To abbreviate notation, we write $\eta_{\star} := \eta_{\star}(u_\star)$ and $\osc_{\star} := \osc_{\star}(u_\star)$.
The following proposition is proved, 
e.g., in~\cite{Carstensen:2005-1,Erath:2013-1};

\begin{proposition}[reliability and efficiency]
\label{prop:reliability-efficiency}
The residual error estimator $\eta_\star$ satisfies
\begin{align}\label{eq:reliableefficient}
 C_{\rm rel}^{-1}\norm{u-u_\star}{H^1(\Omega)}^2 \le \,\eta_\star^2\le\, 
 C_{\rm eff}\big(\norm{u-u_\star}{H^1(\Omega)}^2 + \osc_\star^2),
\end{align}
where $C_{\rm rel},C_{\rm eff}>0$ depend only on 
the $\sigma$-shape regularity  of $\TT_\star$, the data assumptions~\cref{eq:A,eq:bcestimate1}, and $\Omega$. 
\qed
\end{proposition}

Note that a robust variant of this estimator with respect to an energy norm
is found and analyzed in~\cite[Theorem~4.9, Theorem~6.3, 
and Remark~6.1]{Erath:2013-1}, where we additionally require the assumption
$\norm{\div\,\b +\c}{L^{\infty}(\Omega)}\leq C (\frac{1}{2}\div\,\b+\c)$ with $C>0$.
One of the key ingredients of \cref{prop:reliability-efficiency} is~\cref{eq1:lem:discdefect} of 
the following lemma which will be employed below.
The proof of the orthogonality relation~\cref{eq1:lem:discdefect} is well-known and found, e.g.,
in~\cite{Carstensen:2005-1,Erath:2010-phd,Erath:2013-1}, whereas~\cref{eq2:lem:discdefect} is proved
in~\cite[Lemma 16]{Erath:2016-1} for arbitrary refinement of meshes and can easily be transferred
to the present model problem~\cref{eq:model}.

\begin{lemma}
\label{lem:discdefect}
Let $\TT_\diamond\in\refine(\TT_0)$ and  $\TT_{\star} \in \refine(\TT_\diamond)$. Suppose that the
discrete solutions $u_\star\in \SS_0^1(\TT_\star)$ or $u_\diamond\in \SS_0^1(\TT_\star)$ exist.
Then there holds the $L^2$-orthogonality
\begin{align}\label{eq1:lem:discdefect}
    \sum_{T\in\TT_\diamond}\product{R_{\diamond}(u_\diamond)}{v^*_\diamond}_T
    -\sum_{F\in\FF_{\diamond}^\Omega}\product{J_{\diamond}(u_\diamond)}{v^*_\diamond}_F &= 0
  \quad\text{for all }v^*_\diamond\in\PP^0_0(\TT^*_\diamond)\\
\intertext{as well as the discrete defect identity}
\label{eq2:lem:discdefect}
  \sum_{T\in\TT_\diamond}\product{R_{\diamond}(u_\diamond)}{v^*_{\star}}_T - 
  \sum_{F\in\FF_{\diamond}^\Omega}\product{J_{\diamond}(u_\diamond)}{v^*_{\star}}_F
  &= \AA_\star(u_\star-u_\diamond,v^*_{\star})
  \text{ for all }v^*_{\star}\in\PP^0_0(\TT^*_{\star}).
  \qed\hspace*{-3mm}
\end{align}
\end{lemma}

\subsection{Comparison result and \textsl{a~priori} error estimate}
The following proposition states that the FVM error estimator is equivalent to the 
optimal \emph{total error} (i.e., error plus oscillations) and so improves \cref{prop:reliability-efficiency}.
The result is first proved in~\cite{Erath:2016-1} for $\b=0$ and $\c=0$ and generalized to the present model problem in~\cite{Erath:2017-1}.

\begin{proposition}\label{theorem:cea}
Let $\TT_\star$ be sufficiently fine such that
$\Cell-\Cbil\norm{h_\star}{L^\infty(\Omega)}>0$ with $\Cell$ and $\Cbil$ from~\cref{eq:enorm:equivalent} 
and~\cref{eq:bilinearcomp}, respectively. 
Then it holds that
\begin{align}\label{eq:totalerror}
 \begin{split}
 C_1^{-1}\,\eta_\star 
 \le \min_{v_\star\in\SS^1_0(\TT_\star)} \big(\norm{u-v_\star}{H^1(\Omega)} + \osc_\star(v_\star)\big) 
 \le \norm{u-u_\star}{H^1(\Omega)} + \osc_\star
 \le C_1\,\eta_\star.
 \end{split}
\end{align}
Moreover, if $u_\star^{\rm FEM}\in\SS^1_0(\TT_\star)$ denotes the FEM solution of
 $\AA(u_\star^{\rm FEM},w_\star) = (f,w_\star)_\Omega$
 for all $w_\star\in\SS^1_0(\TT_\star)$,
it holds that
\begin{align*}
 C_2^{-1}\,\big(\norm{u-u_\star}{H^1(\Omega)} + \osc_\star\big)
 &\le \norm{u-u_\star^{\rm FEM}}{H^1(\Omega)} + \osc_\star(u_\star^{\rm FEM})\\
 &\le  C_2\,\big(\norm{u-u_\star}{H^1(\Omega)} + \osc_\star\big).
\end{align*}
The constants $C_1,C_2>0$ depend only on $\Omega$, the $\sigma$-shape regularity of $\TT_\star$, 
and the data assumptions~\cref{eq:A,eq:bcestimate1}. \qed
\end{proposition}
As a direct consequence of \cref{theorem:cea}, 
one obtains the following {\sl convergence} result and {\sl a~priori} estimate which confirms 
first-order convergence of FVM; see again~\cite{Erath:2016-1,Erath:2017-1}.
Note that the statement even holds for $u\in H^1_0(\Omega)$, 
whereas in the literature standard FVM analysis usually requires, 
e.g., $u\in H^{1+\eps}(\Omega)$ for some $\eps>0$.

\begin{corollary}\label{corollary:cea}
Let $\{\TT_\star\}$ be a family of sufficiently fine and 
uniformly $\sigma$-shape regular triangulations. Let $u\in H^1_0(\Omega)$ 
be the solution of~\cref{eq:weakform}. Then there holds convergence
\begin{align*}
 \norm{u-u_\star}{H^1(\Omega)} + \osc_\star \to0
 \quad\text{as}\quad
 \norm{h_\star}{L^\infty(\Omega)}\to0.
\end{align*}
Moreover, additional regularity $u\in H^1_0(\Omega)\cap H^2(\Omega)$ implies first-order convergence
\begin{align*}
 \norm{u-u_\star}{H^1(\Omega)} + \osc_\star
 = \OO(\norm{h_\star}{L^\infty(\Omega)}).\qquad\qed
\end{align*}
\end{corollary}


\section{Adaptive FVM}
\label{sec:A3}%

In this section, we apply an adaptive mesh-refining algorithm for FVM. We combine
ideas from~\cite{Mekchay:2005-1} and~\cite{Erath:2016-1} to prove that adaptive FVM 
leads to linear convergence with optimal algebraic rates for the error 
estimator (and hence for the total error;  see \cref{theorem:cea}).

\subsection{Adaptive algorithm}
\label{subsec:algorithm}
%
As in~\cite{Erath:2016-1}, we employ the following adaptive algorithm:

\def\Cmark{C_{\rm mark}}
\begin{algorithm}\label{algorithm:mns}
{\bfseries Input:} Let $0<\theta'\le\theta\le1$ and $\Cmark,\Cmark'\ge1$. Let
$\TT_0$ be a conforming triangulation of $\Omega$ which resolves possible discontinuities of $\A$. \\[1mm]
{\bfseries Loop:} For $\ell=0,1,2,\dots$, iterate the following steps~{\rm(i)--(v)}:
\begin{itemize}
\item[\rm(i)] \textbf{Solve:} Compute the discrete solution 
$u_\ell\in\SS^1_0(\TT_\ell)$ from~\cref{eq:disc_systemfvem}.
\item[\rm(ii)] \textbf{Estimate:} Compute $\eta_\ell(T,u_\ell)$ and
$\osc_\ell(T,u_\ell)$ from~\cref{eq:etaToscT} for all $T\in\TT_\ell$.
\item[\rm(iii)] \textbf{Mark I:} Find $\MM_\ell^\eta\subseteq\TT_\ell$ of up to the multiplicative constant $\Cmark\ge1$ minimal cardinality which satisfies the D\"orfler marking criterion 
\begin{align}\label{eq:doerfler:eta}
 \theta\,\sum_{T\in\TT_\ell}\eta_\ell(T,u_\ell)^2 \le \sum_{T\in\MM_\ell^\eta}\eta_\ell(T,u_\ell)^2.
\end{align}
\item[\rm(iv)] \textbf{Mark II:} Find $\MM_\ell\subseteq\TT_\ell$ of up to the multiplicative constant $\Cmark'\ge1$ minimal cardinality 
which satisfies $\MM_\ell^\eta\subseteq\MM_\ell$ as well as the D\"orfler marking criterion 
\begin{align}\label{eq:doerfler:osc}
 \theta'\,\sum_{T\in\TT_\ell}\osc_\ell(T,u_\ell)^2 \le \sum_{T\in\MM_\ell}\osc_\ell(T,u_\ell)^2.
\end{align}
\item[\rm(v)] \textbf{Refine:} Generate new triangulation $\TT_{\ell+1} := \refine(\TT_\ell,\MM_\ell)$ by refinement of
all marked elements.
\end{itemize}
\smallskip
{\bfseries Output:} Adaptively refined triangulations $\TT_\ell$, 
corresponding discrete solutions $u_\ell$, estimators $\eta_\ell$, 
and data oscillations $\osc_\ell$ for $\ell\ge0$.
\end{algorithm}
Due to the lack of standard Galerkin orthogonality (see \cref{subsec:noGalerkin}),
we additionally have to mark the oscillations~\cref{eq:doerfler:osc}. 
In practice, however, this marking is negligible, since $\theta'$ can be chosen 
arbitrary small;
see~\cite[Remark 7]{Erath:2016-1} for more details.

\subsection{Quasi-Galerkin orthogonality}
\label{subsec:noGalerkin}%
\def\Cddual{C_{\rm dual}}%
Given $g\in L^2(\Omega)$, we consider the dual problem: Find $\phi\in H^1_0(\Omega)$ such that
\begin{align}\label{eq:weakformdual}
 \AA(v,\phi) = \product{g}{v}_\Omega
 \quad\text{for all }v\in H^1_0(\Omega).
\end{align}
The Lax-Milgram theorem proves existence and uniqueness of $\phi\in H^1_0(\Omega)$. 
Let $0< s\leq 1$. We suppose that the dual problem~\cref{eq:weakformdual} 
is $H^{1+s}$-regular, i.e., there exists a constant $\Cddual>0$  
such that for all $g\in L^2(\Omega)$, the solution of~\cref{eq:weakformdual} satisfies
\begin{align}\label{eq:Hs_regular}
 \phi \in H^1_0(\Omega)\cap H^{1+s}(\Omega)
 \quad\text{with}\quad
 \norm{\phi}{H^{1+s}(\Omega)} \le \Cddual\,\norm{g}{L^2(\Omega)}.
\end{align}
We refer to~\cite{grisvard} for a discussion on this regularity assumption.
The main result of this section is the following quasi-Galerkin orthogonality with respect to the operator-induced quasi 
norm from~\cref{eq:enorm:equivalent}. The proof is postponed to the end of this section.

\def\Cgal{C_{\rm gal}}%
\begin{proposition}
\label{lemma:orth}
Let $0< s\leq 1$ and suppose that the dual 
problem~\cref{eq:weakformdual} is $H^{1+s}$-regular~\cref{eq:Hs_regular}.
Let $\TT_\diamond\in\refine(\TT_0)$ and  $\TT_{\star} \in \refine(\TT_\diamond)$.
Then there exists $\Cgal>0$ such that
\begin{align}
 \label{eq:galerkin}
 \begin{split}
\enorm{u-u_\star}^2
&\le\enorm{u-u_\diamond}^2-\frac12\,\enorm{u_\star-u_\diamond}^2
+\Cgal\,\norm{h_\star}{L^\infty(\Omega)}^{2s}\eta_\star^2 
+\Cgal\,\osc_\star^2.
\end{split}
\end{align}
The constant $\Cgal>0$ depends only on $\Cddual$, $\Cosc$, $\Crel$, $\Cell$,
$\CAbil$, $\diam(\Omega)$, and $\norm{\b}{W^{1,\infty}(\Omega)}$ 
as well as on $\sigma$-shape regularity and all possible shapes of element patches in $\TT_\star$.
\end{proposition}

For the FVM error, the classical Galerkin orthogonality fails, i.e.,
$\AA(u-u_\star,v_\star)\not= 0$ for some $v_\star\in\SS^1_0(\TT_\star)$. However, there holds the following estimate; 
see, e.g.,~\cite{Erath:2016-1}.

\begin{lemma}
 \label{lem:galerkinosc}
The FVM error $u-u_\star$ satisfies that
 \begin{align}
  \label{eq:galerkinosc}
  |\AA(u-u_\star,v_\star)|
  \le \Cosc\,\norm{v_\star}{H^1(\Omega)}\,\osc_{\star}
  \quad\text{for all }v_\star\in\SS^1_0(\TT_\star).
 \end{align}
 The constant $\Cosc>0$ depends only on $\sigma$-shape regularity of $\TT_{\star}$.
\end{lemma}

\begin{proof}
Standard calculations (see,
e.g.,~\cite[Theorem~4.9]{Erath:2013-1}) show that
\begin{align*}
  \AA(u-u_\star,v_\star)  
  &=\sum_{T\in\TT_\star}\int_T R_\star(u_\star)\, v_\star\,dx
  +\sum_{F\in\FF_\star^\Omega}\int_{F}J_\star(u_\star)\,v_\star\,ds.
\end{align*}
Together with~\cref{eq1:lem:discdefect} for $v^*_\star=\II_\star^* v_\star\in\PP^0_0(\TT_\star^*)$,
this leads to
\begin{align*}
  \AA(u-u_\star,v_\star) = \sum_{T\in\TT_\star}\int_T R_\star(u_\star)\,(v_\star-v_\star^*)\,dx
  +\sum_{F\in\FF_\star^\Omega}\int_{F}J_\star(u_\star)\,(v_\star-v_\star^*)\,ds.
\end{align*}
We apply~\cref{eq1:proplindual} for the involved integrals and obtain that
\begin{align*}
  \AA(u-u_\star,v_\star) &= \sum_{T\in\TT_\star}\int_T (R_\star(u_\star)-\Pi_{\star}R_\star(u_\star))\,(v_\star-v_\star^*)\,dx\\
  &\qquad+\sum_{F\in\FF_\star^\Omega}\int_{F}(J_\star(u_\star)
  -\Pi_{\star}J_\star(u_\star))\,(v_\star-v_\star^*)\,ds.
\end{align*}
The Cauchy-Schwarz inequality and \cref{eq2:proplindual}--\cref{eq3:proplindual}
conclude the proof.
\end{proof}

\begin{lemma}\label{lemma:dual}
Let $0< s\leq 1$ and suppose that the dual problem~\cref{eq:weakformdual} 
is $H^{1+s}$-regular~\cref{eq:Hs_regular}. Then the FVM error satisfies
 \begin{align}
  \label{eq:L2fvm}
 \Cdual^{-1}\,\norm{u-u_\star}{L^2(\Omega)}^2\le 
 \norm{h_\star}{L^\infty(\Omega)}^{2s}\norm{u-u_\star}{H^1(\Omega)}^2 + \osc_{\star}^2. 
 \end{align}
 The constant $\Cdual>0$ depends only on the $\sigma$-shape regularity of $\TT_{\star}$, $\diam(\Omega)$, $\CAbil$,
 and $\Cddual$ as well as on all possible shapes of element patches in $\TT_\star$.
\end{lemma}

\def\Csz{C_{\rm sz}}
\begin{proof}
The proof is split into two steps.

{\bf Step~1.}
Let $\II_\star:H^1(\Omega)\to\SS^1(\TT_\star)$ be the Scott-Zhang projector~\cite{scottzhang}. Recall the following properties of $\II_\star$ for all $v\in H^1(\Omega)$ 
and $v_\star\in\SS^1(\TT_\star)$, and all $T\in\TT_\star$:
\begin{itemize}
\item $\II_\star$ has a local projection property, i.e., $(\II_\star v)|_T = v_\star|_T$ 
if $v|_{\omega_\star(T)}=v_\star|_{\omega_\star(T)}$;
\item $\II_\star$ preserves discrete boundary data, i.e., $v|_\Gamma = v_\star|_\Gamma$ 
implies that $(\II_\star v)|_\Gamma = v|_\Gamma$;
\item $\II_\star$ is locally $H^1$-stable, i.e., 
$\norm{\nabla\II_\star v}{L^2(T)} \le \Csz\, \norm{\nabla v}{H^1(\omega_\star(T))}$;
\item $\II_\star$ has a local approximation property, 
i.e., $\norm{v-\II_\star v}{L^2(T)}\le \Csz\,h_T\,\norm{\nabla v}{H^1(\omega_\star(T))}$.
\end{itemize}
The constant $\Csz>0$ depends only on $\sigma$-shape regularity of $\TT_\star$.
In particular,
\begin{align*}
\norm{v-\II_\star v}{H^1(\Omega)}\lesssim \norm{v}{H^1(\Omega)}\quad \text{for all } v\in H^1(\Omega),
\end{align*} 
where the hidden constant depends only on $\Csz$ and $\diam(\Omega)$. 
With the local projection property of $\II_\star$, we may apply the Bramble-Hilbert lemma. 
For $v\in H^2(\Omega)$, scaling arguments then prove that 
\begin{align} 
 \norm{v-\II_\star v}{H^1(T)} \lesssim \diam(\omega_\star(T))\,\norm{v}{H^2(\omega_\star(T))}
 \quad\text{for all }T\in\TT_\star,
\end{align}
where the hidden constant depends only on the shape of $\omega_\star(T)$ 
and on the operator norm of $A:=1-\II_\star$ 
(and hence on $\diam(\Omega)$ and $\Csz$). Altogether, this proves the operator norm estimates
\begin{align}\label{eq:interpolation}
 \norm{A:=1-\II_\star:H^{1+t}(\Omega)\to H^1(\Omega)}{}
 \le C\,\norm{h_\star}{L^\infty(\Omega)}^t
 \quad\text{for }t\in\{0,1\},
\end{align}
where $C>0$ depends only on $\Csz$, $\diam(\Omega)$, and all possible shapes of element patches in $\TT_\star$. 
Interpolation arguments~\cite{MR0482275} conclude that~\cref{eq:interpolation} holds for all $0\le t\le 1$.
For $t=s$, this proves that
\begin{align}\label{eq2:interpolation}
 \norm{v-\II_\star v}{H^1(\Omega)} \le C\,\norm{h_\star}{L^\infty(\Omega)}^s\,\norm{v}{H^{1+s}(\Omega)}
 \quad\text{for all }v\in H^{1+s}(\Omega).
\end{align}

{\bf Step~2.}
With $g=v=u-u_\star$ in~\cref{eq:weakformdual}, it holds that
\begin{align*}
 \norm{u-u_\star}{L^2(\Omega)}^2=\AA(u-u_\star,\phi)
 = \AA(u-u_\star,\phi-\II_\star\phi) + \AA(u-u_\star,\II_\star\phi).
\end{align*}
Since we suppose $\phi \in H^{1+s}(\Omega)$,
the first summand is bounded by~\cref{eq2:interpolation}. This yields that
\begin{align*}
 \AA(u-u_\star,\phi-\II_\star\phi)
 &\lesssim \norm{u-u_\star}{H^1(\Omega)}\norm{\phi-\II_\star\phi}{H^1(\Omega)}\\
 &\lesssim \norm{h_\star}{L^\infty(\Omega)}^s\,\norm{u-u_\star}{H^1(\Omega)}\norm{\phi}{H^{1+s}(\Omega)},
\end{align*}
where the hidden constants depends only on $\CAbil$, $\Csz$, and $\diam(\Omega)$.
The second summand is bounded by~\cref{eq:galerkinosc} and $H^1$-stability of $\II_\star$. This yields that
\begin{align*}
 \AA(u-u_\star,\II_\star\phi)
 \lesssim \osc_\star\,\norm{\II_\star\phi}{H^1(\Omega)}
 \lesssim \osc_\star\,\norm{\phi}{H^1(\Omega)}
 \le \osc_\star\,\norm{\phi}{H^{1+s}(\Omega)},
\end{align*}
where the hidden constant depends only on $\Cosc$, $\Csz$ and $\diam(\Omega)$.
Combining the latter three estimates with $H^{1+s}$-regularity~\cref{eq:Hs_regular}, we prove that
\begin{align*}
 \norm{u-u_\star}{L^2(\Omega)}^2
 &\lesssim \big(\norm{h_\star}{L^\infty(\Omega)}^s\,\norm{u-u_\star}{H^1(\Omega)} + \osc_\star\big)\,
 \norm{\phi}{H^{1+s}(\Omega)}
 \\&
 \lesssim \big(\norm{h_\star}{L^\infty(\Omega)}^s\,\norm{u-u_\star}{H^1(\Omega)} + \osc_\star\big)\,\norm{u-u_\star}{L^2(\Omega)},
\end{align*}
where the hidden constant depends additionally on $\Cddual$.
This concludes the proof.
\end{proof}

\begin{proof}[Proof of~\cref{lemma:orth}]
Recall that $\AA(v,w) = \product{\A\nabla v}{\nabla w}_\Omega - \product{\b v}{\nabla w}_\Omega
+\product{c v}{w}_\Omega$ and thus
$\AA(w,v) = \product{\A\nabla w}{\nabla v}_\Omega - \product{\b w}{\nabla v}_\Omega
+\product{c w}{v}_\Omega$.
 For $v,w\in H^1_0(\Omega)$, integration by parts proves that
\begin{align*}
 -\product{\b w}{\nabla v}_\Omega
 = \product{\b \cdot \nabla w}{v}_\Omega
 + \product{\div(\b)\,w}{v}_\Omega
\end{align*}
and hence
\begin{align*}
\AA(v,w)+\AA(w,v)=
2\AA(v,w)+2\product{v}{\b\cdot\nabla w}_\Omega
+\product{\div(\b)\,v}{w}_\Omega.
\end{align*}
By definition of $\enorm\cdot$, this proves that
\begin{align*}
 \enorm{v+w}^2
 &= \enorm{v}^2 + \enorm{w}^2 + \AA(v,w) + \AA(w,v)
 \\&
 = \enorm{v}^2 + \enorm{w}^2 + 2\AA(v,w)+2\product{v}{\b \cdot\nabla w}_\Omega +\product{\div(\b)\,v}{w}_\Omega.
\end{align*}
This leads to
\begin{align*}
\enorm{v}^2 = \enorm{v+w}^2 - \enorm{w}^2 - 2\AA(v,w) - 2\product{v}{\b \cdot\nabla w}_\Omega -\product{\div(\b)\,v}{w}_\Omega.
\end{align*}
With $C_1 := \Cell^{-1}\,(2\norm{\b}{L^\infty(\Omega)}+\norm{\div(\b)}{L^\infty(\Omega)})^2$, 
the Young inequality $ab\le \frac{1}{4}\, a^2 + b^2$ and norm equivalence~\cref{eq:enorm:equivalent} prove that
\begin{align*}
 - 2\product{v}{\b \cdot\nabla w}_\Omega -\product{\div(\b)\,v}{w}_\Omega
 &\le \norm{v}{L^2(\Omega)}\norm{w}{H^1(\Omega)}\,\big(2\norm{\b}{L^\infty(\Omega)}+\norm{\div(\b)}{L^\infty(\Omega)}\big)
 \\&
 \le \frac{1}{4}\,\enorm{w}^2 + C_1\,\norm{v}{L^2(\Omega)}^2.
\end{align*}
Choose $v = u-u_\star$ as well as $w=u_\star-u_\diamond$. So far, we have shown that
\begin{align*}
\enorm{u-u_\star}^2 
&\le \enorm{u-u_\diamond}^2 - \frac{3}{4}\,\enorm{u_\star-u_\diamond}^2 - 2\AA(u-u_\star,u_\star-u_\diamond)
+ C_1\,\norm{u-u_\star}{L^2(\Omega)}^2.
\end{align*}
We apply~\cref{eq:galerkinosc}, 
norm equivalence~\cref{eq:enorm:equivalent}, and the Young inequality $2ab\le \frac{1}{4}\, a^2+4b^2$ to see that
\begin{align*}
&- 2\AA(u-u_\star,u_\star-u_\diamond)
\le 2\,\Cosc\,\norm{u_\star-u_\diamond}{H^1(\Omega)}\,\osc_{\star}
\\&\qquad\qquad
\le 2\,\Cosc\Cell^{-1/2}\,\enorm{u_\star-u_\diamond}\,\osc_{\star}
\le \frac{1}{4}\,\enorm{u_\star-u_\diamond}^2 + 4\,\Cosc^2\Cell^{-1}\,\osc_{\star}^2.
\end{align*}
Next, \cref{lemma:dual} and reliability~\cref{eq:reliableefficient} lead to
\begin{align*}
\Cdual^{-1}\,\norm{u-u_\star}{L^2(\Omega)}^2
\le \norm{h_\star}{L^\infty(\Omega)}^{2s}\,\norm{u-u_\star}{H^1(\Omega)}^2
 +\osc_{\star}^2
\le \Crel\,\norm{h_\star}{L^\infty(\Omega)}^{2s}\,\eta_\star^2
 +\osc_{\star}^2.
\end{align*}
Combining the latter three estimates, we prove that
\begin{align*}
 \enorm{u-u_\star}^2 
 &\le \enorm{u-u_\diamond}^2 - \frac{1}{2}\,\enorm{u_\star-u_\diamond}^2 
  \\&\qquad 
  + C_1\Cdual\Crel\,\norm{h_\star}{L^\infty(\Omega)}^{2s}\,\eta_\star^2
  + \big(4\,\Cosc^2\Cell^{-1} + C_1\Cdual\big)\,\osc_{\star}^2.
\end{align*}
Choosing $\Cgal = \max\{C_1\Cdual\Crel\,,\,4\Cosc^2\Cell^{-1} + C_1\Cdual\big\}$, we conclude the proof.
\end{proof}

\subsection{Linear convergence and general quasi-orthogonality}
\label{subsec:linearconv}
The following properties \cref{eq:A1,eq:A2} of the estimator and~\cref{eq:B1,eq:B2} of the oscillations
are some key observations to prove linear convergence of~\cref{algorithm:mns}.
The proof is based on scaling arguments and can be found in literature, 
e.g.,~\cite[Section~3.1]{ckns} for \cref{eq:A1,eq:A2}
and~\cite[Section~3.3]{Erath:2016-1} for \cref{eq:B1,eq:B2}.
Their proofs apply almost verbatim to the present non-symmetric problem with $\b\not=0$.
Therefore, the details are left to the reader.
\begin{lemma}
There exist constants $0<q<1$ and $C>0$ such that
for all $\TT_\diamond\in\refine(\TT_0)$, all 
$\TT_\star\in\refine(\TT_\diamond)$, and all $v_\star\in\SS_0^1(\TT_\star)$, $v_\diamond\in\SS_0^1(\TT_\diamond)$, it holds that\\
\textbf{(stability of estimator on non-refined elements)}
\begin{align}
 \label{eq:A1}
 \tag{A1}
\Big|\Big(\sum_{T\in\TT_\star\cap\TT_\diamond}\eta_\star(T,v_\star)^2 \Big)^{1/2}
 - \Big(\sum_{T\in\TT_\star\cap\TT_\diamond}\eta_\diamond(T,v_\diamond)^2 \Big)^{1/2}\Big|
 \leq C\, \norm{v_\star-v_\diamond}{H^1(\Omega)},
\end{align}
\textbf{(reduction of estimator on refined elements)}
\begin{align}
 \label{eq:A2}
 \tag{A2}
 \sum_{T\in\TT_\star\backslash\TT_\diamond}\eta_\star(T,v_\star)^2
 \leq  q\sum_{T\in\TT_\diamond\backslash\TT_\star}\eta_\diamond(T,v_\diamond)^2
     + C\, \norm{v_\star-v_\diamond}{H^1(\Omega)}^2,
\end{align}
\textbf{(stability of oscillations on non-refined elements)}
\begin{align}
 \label{eq:B1}
 \tag{B1}
 \begin{split}
\Big|\Big(\sum_{T\in\TT_\star\cap\TT_\diamond}\osc_\star(T,v_\star)^2 \Big)^{1/2}
 &- \Big(\sum_{T\in\TT_\star\cap\TT_\diamond}\osc_\diamond(T,v_\diamond)^2 \Big)^{1/2}\Big|\\
 &\leq C\, \norm{h_\star}{L^\infty(\Omega)}\norm{v_\star-v_\diamond}{H^1(\Omega)},
 \end{split}
\end{align}
\textbf{(reduction of oscillations on refined elements)}
\begin{align}
 \label{eq:B2}
 \tag{B2}
 \sum_{T\in\TT_\star\backslash\TT_\diamond}\osc_\star(T,v_\star)^2
 \leq  q\sum_{T\in\TT_\diamond\backslash\TT_\star}\osc_\diamond(T,v_\diamond)^2
     + C\, \norm{h_\star}{L^\infty(\Omega)}^2\norm{v_\star-v_\diamond}{H^1(\Omega)}^2.
\end{align}
The constants $0<q<1$ and $C>0$ depend only on the $\sigma$-shape regularity~\cref{eq:sigma} 
and on the data assumptions~\cref{eq:A,eq:bcestimate1}.\qed
\end{lemma}

\begin{theorem}[\textbf{linear convergence}]
 \label{theorem:mns:linear}
Let $0<\theta'\le\theta\le1$.
There exists $H>0$ such that the following statement is valid provided that $\norm{h_0}{L^\infty(\Omega)}\le H$ and that the dual problem~\cref{eq:weakformdual} is $H^{1+s}$-regular~\cref{eq:Hs_regular} for some $0< s\leq 1$: There exist $\Clin>0$ and $0<\qlin<1$ such that
\cref{algorithm:mns} guarantees linear convergence in the sense of 
\begin{align}\label{eq:mns:linear}
 \eta_{\ell+n}^2\le\Clin\qlin^n\,\eta_\ell^2
 \quad\text{for all }\ell,n\in\N_0.
\end{align}
The constant $H$ depends only on the $\sigma$-shape regularity~\cref{eq:sigma}, 
on the data assumptions~\cref{eq:A,eq:bcestimate1}, $\Cgal$, $\theta$, and $\theta'$, whereas
$\Clin$ and $\qlin$ additionally depend on $\CAbil$ and $\Crel$.
\end{theorem}

\def\qref{q_{\rm ref}}
\def\qosc{q_{\rm osc}}
\begin{proof}
The proof is split into three steps.

{\bf Step 1.}
There exist constants $C>0$ and $0<q<1$ which depend only on $0<\theta\le1$, $\Cell$, and the constants in~\cref{eq:A1,eq:A2}, such that
\begin{align}\label{eq:mns:step1}
 \eta_{\ell+1}^2 \le q\,\eta_\ell^2 + C \, \enorm{u_{\ell+1}-u_\ell}^2
 \quad\text{for all }\ell\in\N_0.
\end{align}
Furthermore, there exist constants $C>0$ and $0<q<1$ 
which depend only on $0<\theta'\le 1$, $\Cell$, and the constants in~\cref{eq:B1,eq:B2}, such that
\begin{align}\label{eq:mns:step2}
 \osc_{\ell+1}^2 \le q\,\osc_\ell^2 + C\,\norm{h_{\ell+1}}{L^\infty(\Omega)}^2 \,
 \enorm{u_{\ell+1}-u_\ell}^2
 \quad\text{for all }\ell\in\N_0:
\end{align}
The proofs of~\cref{eq:mns:step1} and~\cref{eq:mns:step2} 
rely only on~\cref{eq:A1,eq:A2} with the D\"orfler marking~\cref{eq:doerfler:eta} and~\cref{eq:B1,eq:B2}  
with marking~\cref{eq:doerfler:osc}, respectively.
For details, we refer, e.g., to~\cite[Proposition 10 (step~1 and step~2)]{Erath:2016-1}.

{\bf Step 2.}
Without loss of generality, we may assume that the 
constants $C>0$ and $0<q<1$ in~\cref{eq:mns:step1}--\cref{eq:mns:step2} are the same.
With free parameters $\gamma,\mu>0$, 
we define 
\begin{align*}
 \Delta_\star := \enorm{u-u_{\star}}^2 + \gamma\,\eta_\star^2 + \mu\,\osc_\star^2.
\end{align*}
We employ the quasi-Galerkin orthogonality~\cref{eq:galerkin} and obtain that
\begin{align*}
\Delta_{\ell+1}
 \le \enorm{u-u_\ell}^2
 + [\gamma+\Cgal\,\norm{h_{\ell+1}}{L^\infty(\Omega)}^{2s}\big]\eta_{\ell+1}^2 
 + [\mu + \Cgal]\,\osc_{\ell+1}^2
 - \frac12\,\enorm{u_{\ell+1}-u_\ell}^2.
\end{align*}
Using~\cref{eq:mns:step1}--\cref{eq:mns:step2}, we further derive that
\begin{align*}
 \Delta_{\ell+1}
 &\le \enorm{u-u_\ell}^2
 + \big[\gamma+\Cgal\,\norm{h_{\ell+1}}{L^\infty(\Omega)}^{2s}\big]\,q\,\eta_{\ell}^2
 + \big[\mu + \Cgal\big]\,q\,\osc_{\ell}^2\\
 &\qquad-\bigg(\frac12-C\,\big[\gamma+\Cgal\,\norm{h_{\ell+1}}{L^\infty(\Omega)}^{2s}\big]
 - C\,\norm{h_{\ell+1}}{L^\infty(\Omega)}^2\,\big[\mu + \Cgal\big]\bigg)\,\enorm{u_{\ell+1}-u_\ell}^2.
\end{align*}
Let $H>0$ be a free parameter and suppose that $\norm{h_0}{L^\infty(\Omega)} \le H$.
We estimate $\norm{h_{\ell+1}}{L^\infty(\Omega)} \le \norm{h_0}{L^\infty(\Omega)} \le H$. Norm
equivalence~\cref{eq:enorm:equivalent} and reliability~\cref{eq:reliableefficient} prove that
\begin{align*}
 \enorm{u-u_\ell}^2 \le \CAbil\, \norm{u-u_\ell}{H^1(\Omega)}^2 \le \CAbil\Crel\,\eta_\ell^2.
\end{align*}
Let $\eps>0$ be a free parameter. Combining the last two estimates, we see that
\begin{align*}
\Delta_{\ell+1}
  &\le (1-\eps)\,\enorm{u-u_\ell}^2
 + \gamma\,\big[(1+\gamma^{-1}\Cgal\,H^{2s})q+\gamma^{-1}\eps\,\CAbil\Crel\big]\,\eta_{\ell}^2
 + \mu\,\big[1 + \mu^{-1}\Cgal\big]\,q\,\osc_{\ell}^2\\
 &\qquad-\bigg(\frac12-C\,\big[\gamma+\Cgal\,H^{2s}\big]
 - C\,H^2\big[\mu + \Cgal\big]\bigg)\,\enorm{u_{\ell+1}-u_\ell}^2.
\end{align*}

{\bf Step 3.} It only remains to fix the four free parameters $\gamma$, $\mu$, $\eps$, and $H$.
\begin{itemize}
\item Choose $\gamma>0$ sufficiently small such that $\gamma C < 1/2$.
\item Choose $\mu>0$ sufficiently large such that $\qosc := \big[1 + \mu^{-1}\Cgal\big]\,q < 1$.
\item Choose $H$ sufficiently small such that 
\begin{itemize}
\item[$\bullet$] $C\,\big[\gamma+\Cgal\,H^{2s}\big] + C\,H^2\big[\mu + \Cgal\big] < 1/2$,
\item[$\bullet$] $(1+\gamma^{-1}\Cgal\,H^{2s}) q < 1$.
\end{itemize}
\item Choose $0<\eps< 1$ such that
$\qest := \big[(1+\gamma^{-1}\Cgal\,H^{2s})q+\gamma^{-1}\eps\,\CAbil\Crel\big] < 1$.
\end{itemize}
With $\qlin:=\max\{\,1-\eps\,,\,\qest\,,\,\qosc\,\}$, we then obtain that
\begin{align*}
 \Delta_{\ell+1}
 &\le (1-\eps)\,\enorm{u-u_\ell}^2
 + \gamma\,\big[(1+\gamma^{-1}\Cgal\,H^{2s})q+\gamma^{-1}\eps\,\CAbil\Crel\big]\,\eta_{\ell}^2
 + \mu\,\big[1 + \mu^{-1}\Cgal\big]\,q\,\osc_{\ell}^2
 \\&
 \le \max\{\,1-\eps\,,\,\qest\,,\,\qosc\,\}\,\Delta_\ell=\qlin\Delta_\ell.
\end{align*}
Induction on $n$, norm
equivalence~\cref{eq:enorm:equivalent},
reliability~\cref{eq:reliableefficient}, and $\osc_{\ell}^2\leq\eta_\ell^2$ prove that
\begin{align*}
 \gamma \,\eta_{\ell+n}^2\leq\Delta_{\ell+n}\leq\qlin^n\Delta_\ell
 \leq \qlin^n(\Crel\CAbil+\gamma+\mu)\,\eta_\ell^2\quad \text{for all }\ell,n\in\N_0
\end{align*}
This concludes linear convergence~\cref{eq:mns:linear} 
with $\Clin=(\Crel\CAbil+\gamma+\mu)\gamma^{-1}$.
\end{proof} 

From the linear convergence~\cref{eq:mns:linear}, we immediately obtain the so-called
\emph{general quasi-orthogonality}; see, e.g.,~\cite[Proposition~4.11]{axioms} or \cite[Proposition 10 (step~5)]{Erath:2016-1}. 

\begin{corollary}[\textbf{general quasi-orthogonality}]
 Let $(u_k)$ be the sequence of solutions of \cref{algorithm:mns}.
 Then there exists $C>0$ such that
\begin{align}
 \label{eq:A3}
 \tag{A3}
 \sum_{k=\ell}^\infty \norm{u_{k+1}-u_k}{H^1(\Omega)}^2
 \le C\,\eta_\ell^2
 \quad \text{ for all $\ell\in\N_0$.}
\end{align}
The constant $C>0$ has the same dependencies as $\Clin$ from~\cref{eq:mns:linear}.
\end{corollary}

\subsection{Optimal algebraic convergence rates}
In order to prove optimal convergence rates of \cref{algorithm:mns}, we need one further property
of the error estimator, namely the so-called \emph{discrete reliability} (A4). The proof of the following lemma follows as for the symmetric case in~\cite[Proposition 15]{Erath:2016-1}. While the proof is thus omitted, we note that the main difficulties over the well-known FEM proof~\cite{ckns} arise in the handling 
of the piecewise constant test space on $\TT_\star^*$ and $\TT_\diamond^*$, respectively,
and the fact that these test spaces are not nested.

\begin{lemma}[\textbf{discrete reliability}]
There exists a constant $C>0$ such that for
all $\TT_\diamond\in\refine(\TT_0)$ and all 
$\TT_\star\in\refine(\TT_\diamond)$, it holds that
\begin{align}
\label{eq:A4}
\tag{A4}
 \norm{u_\star - u_\diamond}{H^1(\Omega)}^2
 \le C\Big(\sum_{T\in\TT_{\star}}h_T^2\norm{u_\star - u_\diamond}{H^1(T)}^2+
 \sum_{T\in\RR_\diamond}\eta_{\diamond}(T,u_\diamond)^2\Big),
\end{align}
where
$\RR_\diamond:=\set{T\in\TT_\diamond}{\exists T'\in\TT_\diamond\backslash\TT_{\star}
\text{ with }T\cap T'\neq\emptyset}$
consists of all refined elements $\TT_\diamond\backslash\TT_{\star}$ plus one
additional layer of neighboring elements.
The constant $C>0$ depends only on the $\sigma$-shape regularity~\cref{eq:sigma}, 
the data assumptions~\cref{eq:A,eq:bcestimate1}, and $\Omega$.
Note that for a sufficiently fine initial mesh $\TT_0$, e.g., $C\,\norm{h_0}{L^\infty(\Omega)}^2 \leq 1/2$,~\cref{eq:A4} 
leads to discrete reliability as stated in~\cite{axioms}.\hfill\qed
\end{lemma}

Let $\T := \refine(\TT_0)$ be the set of all possible triangulations obtained by NVB. 
For $N\ge0$, let $\T_N := \set{\TT_\star\in\T}{\#\TT_\star-\#\TT_0 \le N}$. For $s>0$, define
\begin{align}
 \norm{u}{\mathbb{A}_s} := \sup_{N\in\N_0} \inf_{\TT_\star\in\T_N} (N+1)^s\,\eta_\star.
\end{align}
Note that $\norm{u}{\mathbb{A}_s} < \infty$ implies an algebraic decay $\eta_\star = \OO\big((\#\TT_\star)^{-s}\big)$ along the optimal sequence of meshes (which minimize the error estimator). Optimal convergence of the adaptive algorithm thus means that for all $s>0$ with $\norm{u}{\mathbb{A}_s} < \infty$, the adaptive algorithm leads to $\eta_\ell = \OO\big((\#\TT_\ell)^{-s}\big)$.
The work~\cite[Theorem~4.1]{axioms} proves in a general framework the following \cref{theorem:mns}, 
if the adaptive algorithm applied to a numerical scheme and a corresponding 
estimator satisfies~\cref{eq:A1},~\cref{eq:A2},~\cref{eq:A3}, and~\cref{eq:A4}.

\begin{theorem}[\textbf{optimal algebraic convergence rates}]
 \label{theorem:mns}
Suppose that the dual problem~\cref{eq:weakformdual} 
is $H^{1+s}$-regular~\cref{eq:Hs_regular} for some $0< s\leq 1$. Let
the initial mesh $\TT_0$ be sufficiently fine, i.e, there exists a constant $H>0$
such that $\norm{h_0}{L^\infty(\Omega)}\le H$.
Finally, suppose that there is a constant $\Cmns\ge 1$ such that $\#\MM_\ell\le\Cmns\#\MM_\ell^\eta$
for all $\ell\in\N_0$.
Then there  exists a bound $0<\theta_{\rm opt}\le 1$ such that for all $0<\theta<\theta_{\rm opt}$ and all $s>0$ with $\norm{u}{\mathbb{A}_s}<\infty$, there exists a constant $\Copt>0$
such that
\begin{align}\label{eq:mns:optimal}
\eta_\ell\leq \Copt(\#\TT_\ell-\#\TT_0)^{-s}
\quad \text{for all } \ell \in \N.
\end{align}
The constant $\theta_{\rm opt}$ depends only on $\Omega$, $H$, uniform $\sigma$-shape regularity of the triangulations 
$\TT_\star\in\refine(\TT_0)$, and the data assumptions~\cref{eq:A,eq:bcestimate1}. The constant
$\Copt$ additionally depends on $s$, the constant $\qlin$ from~\cref{eq:mns:linear}, the use of NVB, and on $\Cmns$.\hfill \qed
\end{theorem}
\begin{remark}
 \label{rem:marking}
 A direct consequence of the assumption $\#\MM_\ell\le\Cmns\#\MM_\ell^\eta$
 in \cref{theorem:mns} is
 that data oscillation marking~\eqref{eq:doerfler:osc} 
 is negligible with respect to the overall number of marked elements~\cite[Remark 7]{Erath:2016-1}. 
 In practice,~\cref{eq:doerfler:eta} already implies~\cref{eq:doerfler:osc}
 since $\theta'>0$ can be chosen arbitrarily small. 
 Furthermore, efficiency~\cref{eq:reliableefficient}
 is not required to show~\cref{eq:mns:linear} and \cref{eq:mns:optimal} 
 but guarantees (optimal) linear convergence 
 also for the FVM error.
\end{remark}
%
%

\section{Numerical examples}
\noindent
In extension of our theory, we consider
the model problem~\eqref{eq:model} with inhomogeneous Dirichlet boundary conditions. 
For all experiments in 2D, we run \cref{algorithm:mns} 
with $\theta=1=\theta'$ and $\theta=0.5=\theta'$ for uniform mesh-refinement
and adaptive mesh-refinement, respectively. 

\begin{figure}
\begin{center}
	\subfigure[\label{subfig:bsp1meshT16}$\TT_{16}$ (3842 elements).]
	{\includegraphics[width=0.3\textwidth]
	{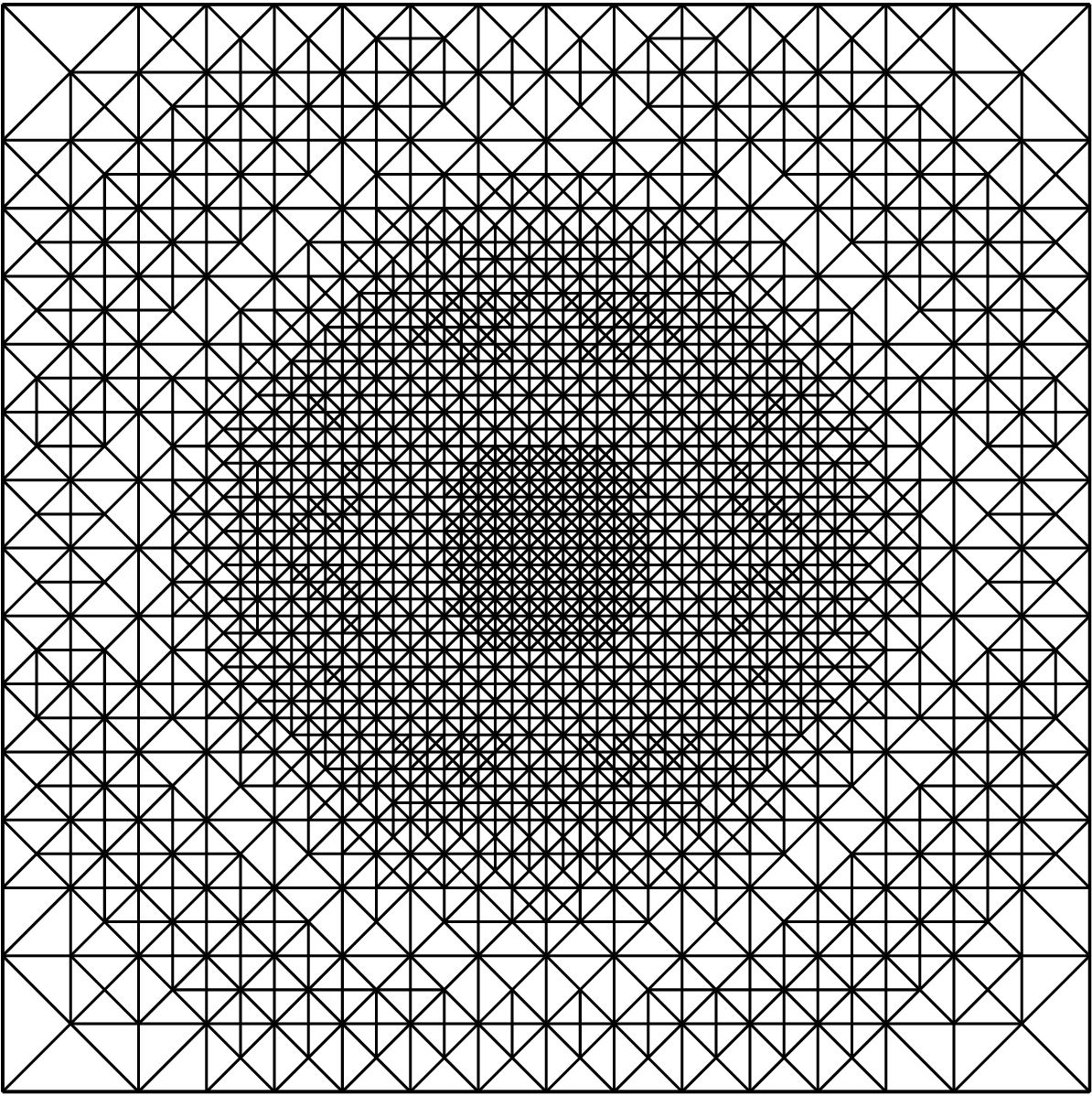}}
	\hspace{0.05\textwidth}
	\subfigure[\label{subfig:bsp1sol}Solution ($\TT_{16}$).]
	{\includegraphics[width=0.46\textwidth]
	{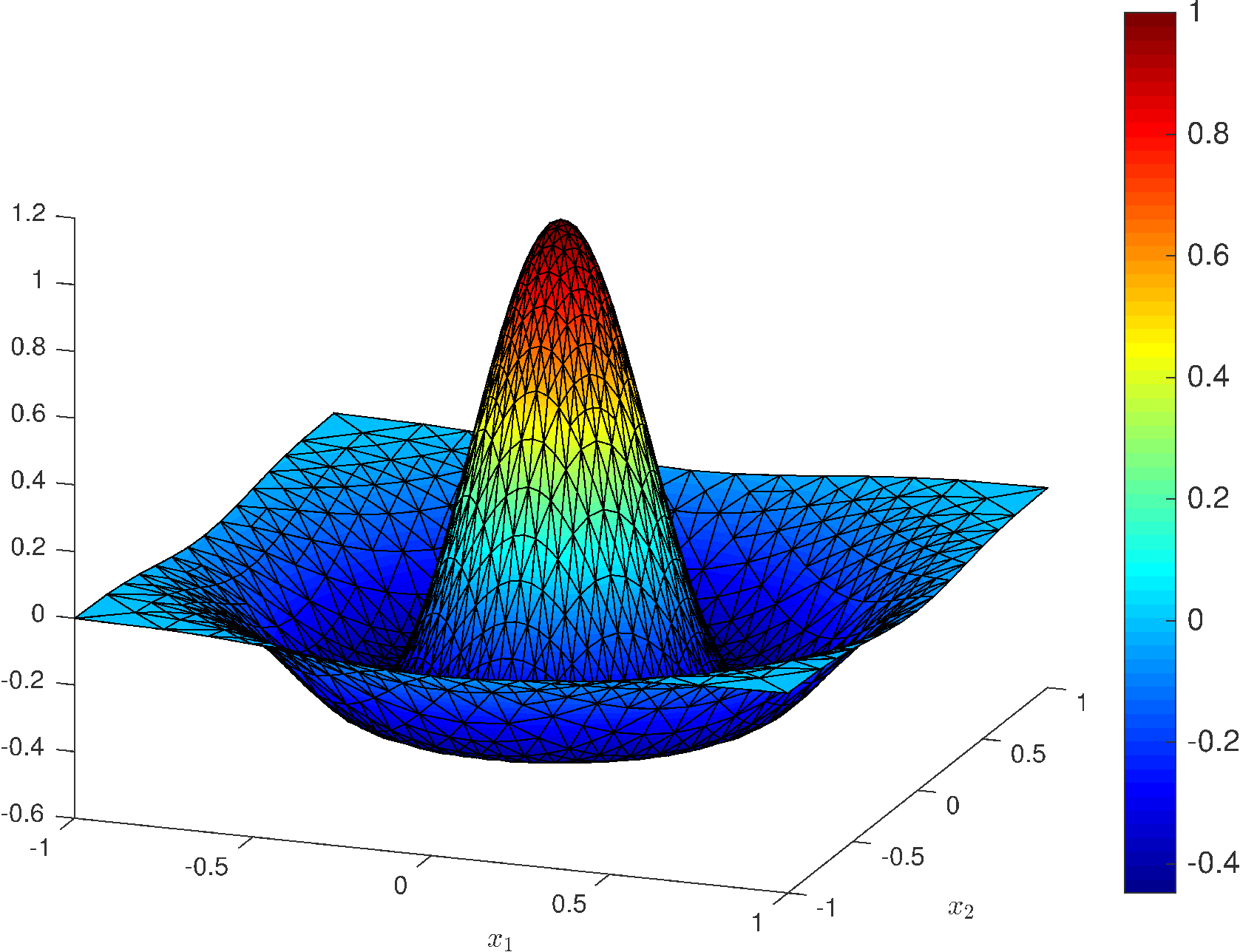}}
\end{center}
\caption{\label{fig:bsp1meshsol}%
Experiment with smooth solution from \cref{ex:bsp1}: 
Adaptively generated mesh $\TT_{16}$ from a
uniform initial triangulation $\TT_0$ with $16$ elements, and discrete FVM solution calculated on $\TT_{16}$.}
\end{figure}
\begin{figure}
\begin{center}
\begin{psfrags}%
\psfragscanon%
%
\psfrag{s01}[l][l]{\scriptsize$1$}%
\psfrag{s08}[l][l]{\small $\eta_\ell$ (uni.)}%
\psfrag{s09}[l][l]{\small $\osc_\ell$ (ada.)}%
\psfrag{s10}[l][l]{\scriptsize$1$}%
\psfrag{s07}[l][l]{\small $E_\ell$ (uni.)}%
\psfrag{s04}[l][l]{\small $\eta_\ell$ (ada.)}%
\psfrag{s02}[l][l]{\small $\osc_\ell$ (uni.)}%
\psfrag{s06}[l][l]{\small $E_\ell$ (ada.)}%
\psfrag{s03}[l][l]{\scriptsize$\;\;\;\;\;\;\;\;1$}%
\psfrag{s14}[l][l]{\scriptsize$\;\;\;1/2$}%
\psfrag{s05}[t][t]{\small number of elements}%
\psfrag{s11}[b][b]{\small error, estimator, oscillation}%
%
\color[rgb]{0.15,0.15,0.15}%
%
\psfrag{x01}[t][t]{\scriptsize${10^{1}}$}%
\psfrag{x02}[t][t]{\scriptsize${10^{2}}$}%
\psfrag{x03}[t][t]{\scriptsize${10^{3}}$}%
\psfrag{x04}[t][t]{\scriptsize${10^{4}}$}%
\psfrag{x05}[t][t]{\scriptsize${10^{5}}$}%
\psfrag{x06}[t][t]{\scriptsize${10^{6}}$}%
\psfrag{x07}[t][t]{\scriptsize${10^{7}}$}%
%
\psfrag{v01}[r][r]{\scriptsize${10^{-3}}$}%
\psfrag{v02}[r][r]{\scriptsize${10^{-2}}$}%
\psfrag{v03}[r][r]{\scriptsize${10^{-1}}$}%
\psfrag{v04}[r][r]{\scriptsize${10^{0}}$}%
\psfrag{v05}[r][r]{\scriptsize${10^{1}}$}%
\psfrag{v06}[r][r]{\scriptsize${10^{2}}$}%
\psfrag{v07}[r][r]{\scriptsize${10^{3}}$}%
%
\includegraphics[width=0.8\textwidth]{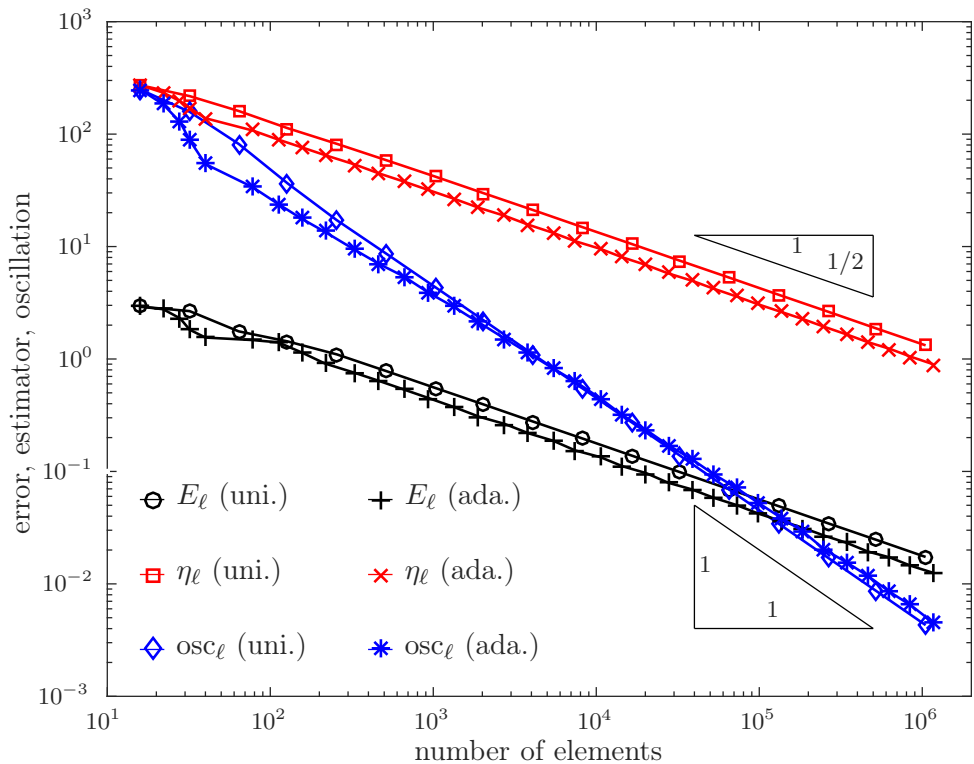}%
\end{psfrags}%
\end{center}
\caption{\label{fig:bsp1error}%
Experiment with smooth solution from \cref{ex:bsp1}: 
Error $E_\ell=\norm{u-u_\ell}{H^1(\Omega)}$, weighted-residual error estimator
$\eta_\ell$, and data oscillations $\osc_\ell$ for uniform and adaptive
mesh-refinement.}
\end{figure}
\subsection{Experiment with smooth solution}
\label{ex:bsp1}
On the square $\Omega=(-1,1)^2$, 
we prescribe
the exact solution $u(x_1,x_2) = (1-10x_1^2-10x_2^2)e^{-5(x_1^2+x_2^2)}$
with $x=(x_1,x_2)\in\R^2$.
We choose the diffusion matrix 
\begin{align*}
	\A=  \left ( \begin{array}{rr}
  10+\cos x_1 & 9\,x_1 x_2 \\
  9\,x_1 x_2 & \;10+\sin x_2
  \end{array}\right),
\end{align*}
the velocity $\b=(\sin x_1,\cos x_2)^T$ and the reaction $\c=1$.
Note that~\cref{eq:A} holds with $\lambda_{\min}=0.82293$ and
$\lambda_{\max}=10.84096$, and~\cref{eq:bcestimate1} with $\frac{1}{2} \div\b+\c> 0$.
The right-hand side $f$ is calculated appropriately.
The uniform initial mesh $\TT_0$ consists of the $16$ triangles. 

In \cref{fig:bsp1meshsol}\subref{subfig:bsp1meshT16} we see
an adaptively generated mesh after $16$ refinements.
\cref{fig:bsp1meshsol}\subref{subfig:bsp1sol} plots the 
smooth solution on the mesh $\TT_{16}$.
Both, uniform and adaptive mesh-refinement, lead 
to the optimal convergence order $\OO(N^{-1/2})$ with respect to the number $N$ of elements
since $u$ is smooth; see 
\cref{fig:bsp1error}.
The oscillations are of higher order and decrease with $\OO(N^{-1})$.

\cref{tab:bsp1} shows the experimental validation of the additional 
assumption in \cref{theorem:mns}, i.e., marking for the data oscillations is negligible;
see also \cref{rem:marking}. 

\begin{table}[!t]\small
\begin{center}
\begin{tabular}[t]{r|rcc}
\hline 
$\ell$  &  $\#\TT_\ell$  & $\frac{\#\MM_\ell}{\#\MM_\ell^\eta} $ 
& $ \frac{\osc_{\ell}(\MM_\ell^\eta)^2}{\osc_\ell^2} $ \\
\hline 
\hline 
  0 &       16 &    1.000 &      0.631 \\  
  1 &       22 &    1.000 &      0.615 \\  
  2 &       28 &    1.000 &      0.704 \\  
  3 &       32 &    1.000 &      0.769 \\  
  4 &       40 &    1.214 &      0.338 \\  
  5 &       78 &    1.111 &      0.446 \\  
  6 &      112 &    1.133 &      0.292 \\  
  7 &      156 &    1.119 &      0.410 \\  
  8 &      216 &    1.062 &      0.394 \\  
  9 &      331 &    1.198 &      0.264 \\  
 10 &      460 &    1.014 &      0.472 \\  
 11 &      660 &    1.049 &      0.371 \\  
\hline
\end{tabular}%
\hfill%
\begin{tabular}[t]{r|rcc}
\hline 
$\ell$  &  $\#\TT_\ell$  & $\frac{\#\MM_\ell}{\#\MM_\ell^\eta} $ 
& $ \frac{\osc_{\ell}(\MM_\ell^\eta)^2}{\osc_\ell^2} $ \\
\hline 
\hline 
 12 &      944 &    1.027 &      0.431 \\  
 13 &     1,338 &    1.025 &      0.400 \\  
 14 &     1,910 &    1.018 &      0.387 \\  
 15 &     2,748 &    1.026 &      0.374 \\  
 16 &     3,842 &    1.015 &      0.358 \\  
 17 &     5,430 &    1.003 &      0.449 \\  
 18 &     7,438 &    1.013 &      0.359 \\  
 19 &    10,590 &    1.003 &      0.445 \\  
 20 &    14,478 &    1.019 &      0.323 \\  
 21 &    20,286 &    1.004 &      0.430 \\  
 22 &    27,558 &    1.004 &      0.457 \\  
 23 &    38,450 &    1.010 &      0.324 \\  
\hline
\end{tabular}%
\hfill%
\begin{tabular}[t]{r|rcc}
\hline 
$\ell$  &  $\#\TT_\ell$  & $\frac{\#\MM_\ell}{\#\MM_\ell^\eta} $ 
& $ \frac{\osc_{\ell}(\MM_\ell^\eta)^2}{\osc_\ell^2} $ \\
\hline 
\hline 
 24 &    52,422 &    1.000 &      0.540 \\  
 25 &    72,454 &    1.007 &      0.404 \\  
 26 &    98,232 &    1.000 &      0.508 \\  
 27 &   135,172 &    1.004 &      0.446 \\  
 28 &   184,142 &    1.000 &      0.606 \\  
 29 &   251,896 &    1.002 &      0.475 \\  
 30 &   342,148 &    1.001 &      0.488 \\  
 31 &   461,674 &    1.000 &      0.617 \\  
 32 &   635,266 &    1.004 &      0.416 \\  
 33 &   852,730 &    1.000 &      0.664 \\  
 34 &  1,172,122 &    1.002 &      0.464 \\  
\hline 
\end{tabular}
\end{center}
\caption{Experiment with smooth solution from \cref{ex:bsp1}: 
We compute $\widetilde{C}_{MNS}:=\#\MM_\ell/\#\MM_\ell^\eta \le 1.3$. 
Hence, the additional assumption in \cref{theorem:mns} is experimentally verified. 
Furthermore, 
we compute $\widetilde\theta':=\osc_\ell(\MM_\ell^\eta)^2/\osc_\ell^2 \ge 0.2$
with $\osc_{\ell}(\MM_\ell^\eta)^2:=\sum_{T\in\MM_\ell^\eta} \osc_\ell(T,u_\ell)^2$, i.e., 
the choice $\theta = 0.5$, $\theta' = 0.2$ would guarantee $\MM_\ell = \MM_\ell^\eta$ in \cref{algorithm:mns}.}
\label{tab:bsp1}
\end{table}


\begin{figure}
\begin{center}
	\subfigure[\label{subfig:bsp2meshT16}$\TT_{16}$ (2534 elements).]
	{\includegraphics[width=0.3\textwidth]
	{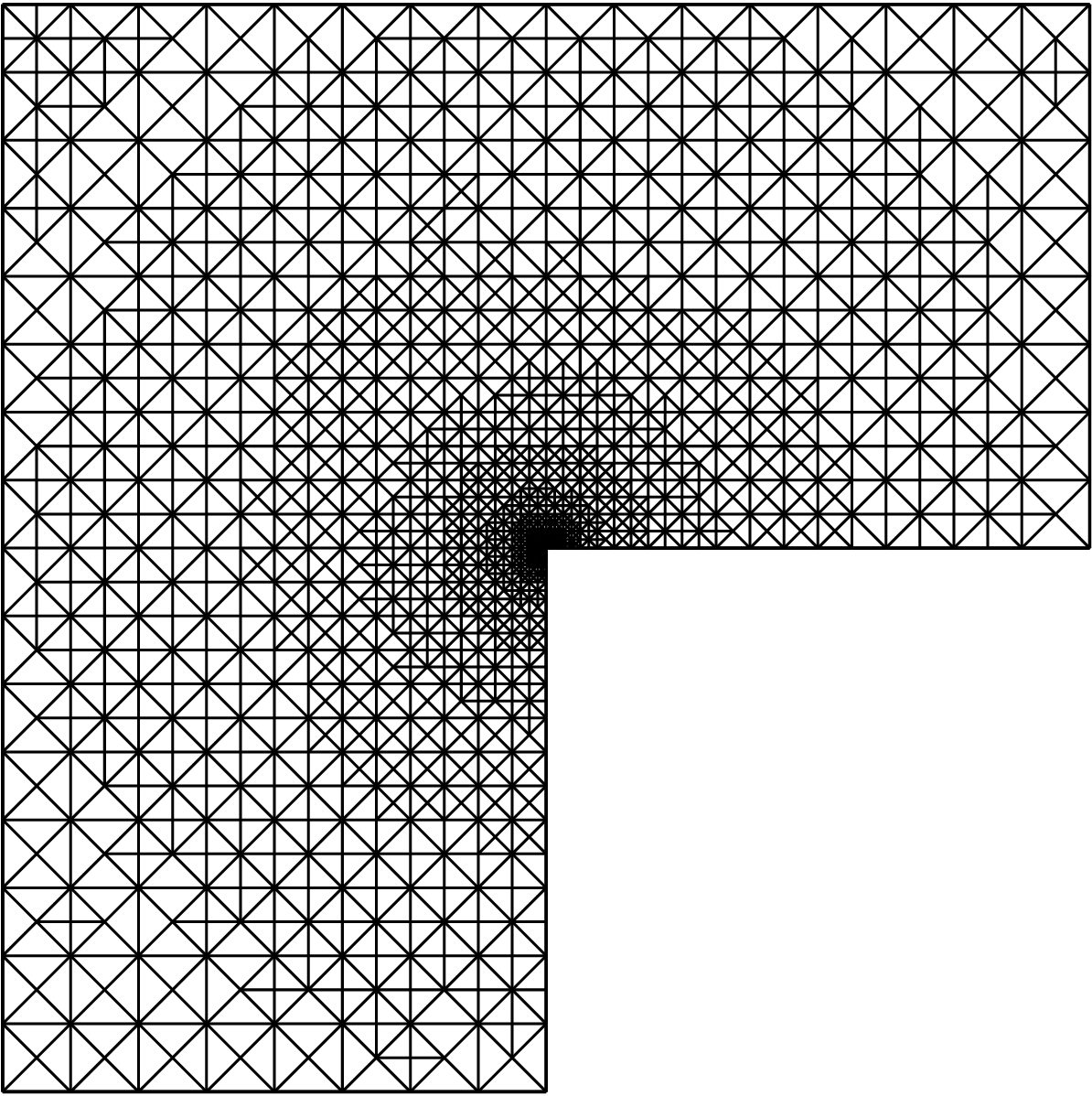}}
	\hspace{0.05\textwidth}
	\subfigure[\label{subfig:bsp2sol}Solution ($\TT_{16}$).]
	{\includegraphics[width=0.46\textwidth]
	{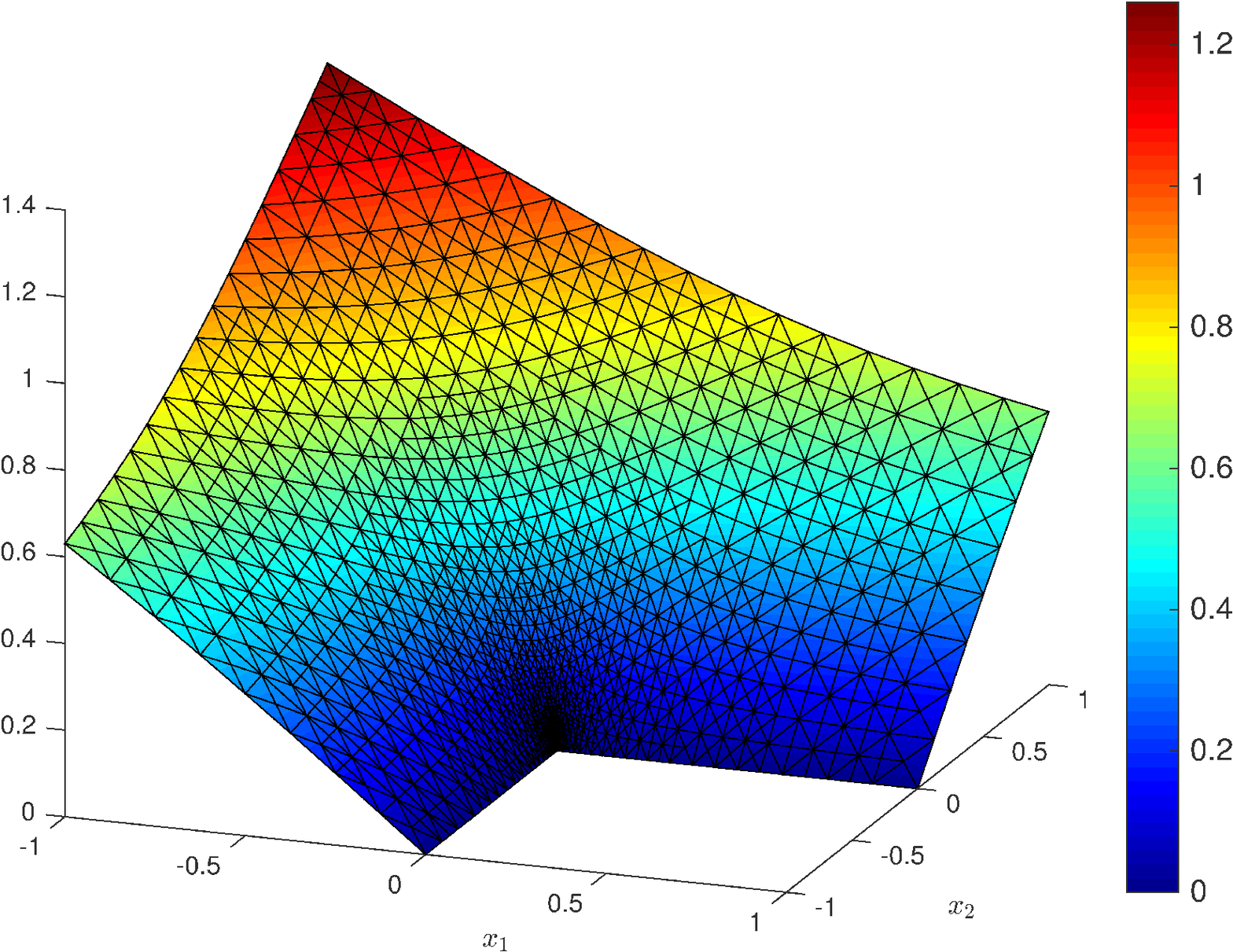}}
\end{center}
\caption{\label{fig:bsp2meshsol}%
Experiment with generic singularity in the reentrant corner $(0,0)$ from \cref{ex:bsp2}:
Adaptively generated mesh $\TT_{16}$ from a
uniform initial triangulation $\TT_0$ with $12$ elements,
and discrete FVM solution calculated on $\TT_{16}$.}
\end{figure}
\begin{figure}
\begin{center}
\begin{psfrags}%
\psfragscanon%
%
\psfrag{s08}[l][l]{\small $\osc_\ell$ (uni.)}%
\psfrag{s15}[t][t]{\small number of elements}%
\psfrag{s05}[l][l]{\small $E_\ell$ (uni.)}%
\psfrag{s16}[l][l]{\small $\eta_\ell$ (uni.)}%
\psfrag{s12}[l][l]{\small $E_\ell$ (ada.)}%
\psfrag{s13}[b][b]{\small error, estimator, oscillation}%
\psfrag{s10}[l][l]{\scriptsize $\;\;\;\;\;\;\;\;\;\;1/2$}%
\psfrag{s14}[l][l]{\small $\osc_\ell$ (ada.)}%
\psfrag{s02}[l][l]{\scriptsize $1$}%
\psfrag{s03}[l][l]{\scriptsize $\;\;\;\;\;\;\;\;\;\;1$}%
\psfrag{s04}[l][l]{\scriptsize $\;\;\;1/3$}%
\psfrag{s09}[l][l]{\small $\eta_\ell$ (ada.)}%
\psfrag{s06}[l][l]{\scriptsize $1$}%
\psfrag{s07}[l][l]{\scriptsize $1$}%
%
\color[rgb]{0.15,0.15,0.15}%
%
\psfrag{x01}[t][t]{\scriptsize${10^{1}}$}%
\psfrag{x02}[t][t]{\scriptsize${10^{2}}$}%
\psfrag{x03}[t][t]{\scriptsize${10^{3}}$}%
\psfrag{x04}[t][t]{\scriptsize${10^{4}}$}%
\psfrag{x05}[t][t]{\scriptsize${10^{5}}$}%
\psfrag{x06}[t][t]{\scriptsize${10^{6}}$}%
\psfrag{x07}[t][t]{\scriptsize${10^{7}}$}%
%
\psfrag{v01}[r][r]{\scriptsize${10^{-5}}$}%
\psfrag{v02}[r][r]{\scriptsize${10^{-4}}$}%
\psfrag{v03}[r][r]{\scriptsize${10^{-3}}$}%
\psfrag{v04}[r][r]{\scriptsize${10^{-2}}$}%
\psfrag{v05}[r][r]{\scriptsize${10^{-1}}$}%
\psfrag{v06}[r][r]{\scriptsize${10^{0}}$}%
\psfrag{v07}[r][r]{\scriptsize${10^{1}}$}%
%
\includegraphics[width=0.8\textwidth]{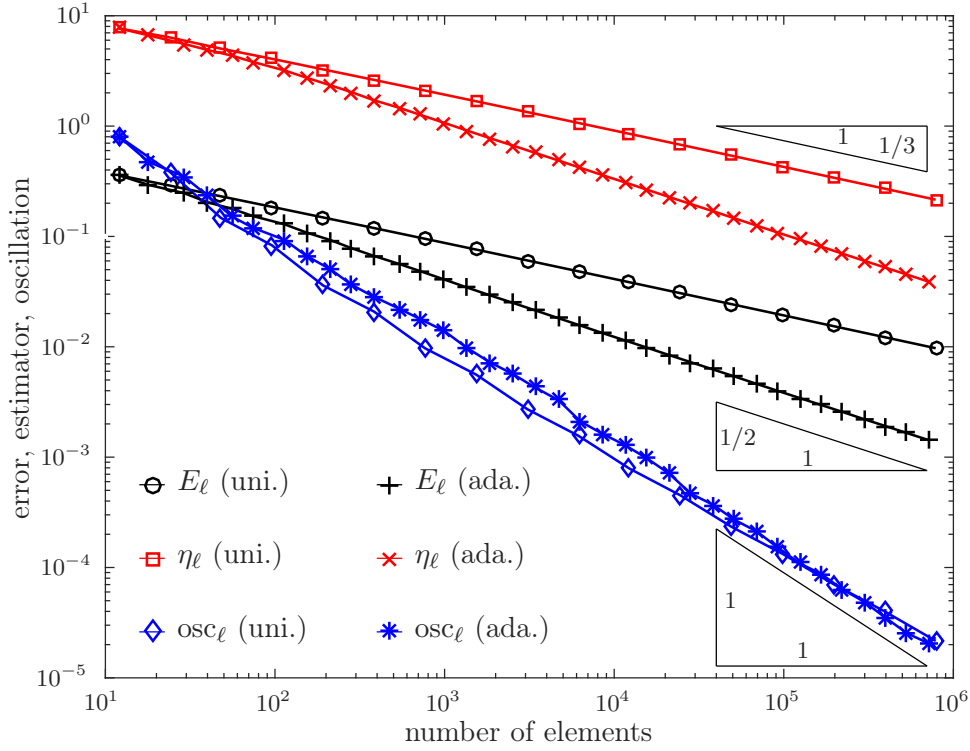}%
\end{psfrags}%
\end{center}
\caption{\label{fig:bsp2error}%
Experiment with generic singularity from \cref{ex:bsp2}: 
Error $E_\ell=\norm{u-u_\ell}{H^1(\Omega)}$, weighted-residual error estimator
$\eta_\ell$, and data oscillations $\osc_\ell$ for uniform and adaptive
mesh-refinement.}
\end{figure}

\begin{table}[!t]\small
\begin{center}
\begin{tabular}[t]{r|rcc}
\hline 
$\ell$  &  $\#\TT_\ell$  & $\frac{\#\MM_\ell}{\#\MM_\ell^\eta} $ & $ \frac{\osc_\ell(\MM_\ell^\eta)^2}{\osc_\ell^2} $ \\
\hline 
\hline 
  0 &       12 &       1.667 &      0.135 \\  
  1 &       18 &       1.750 &      0.086 \\  
  2 &       29 &       1.600 &      0.027 \\  
  3 &       40 &       1.375 &      0.057 \\  
  4 &       56 &       1.400 &      0.252 \\  
  5 &       74 &       1.667 &      0.079 \\  
  6 &      114 &       1.286 &      0.148 \\  
  7 &      153 &       1.188 &      0.243 \\  
  8 &      212 &       1.111 &      0.256 \\  
  9 &      284 &       1.065 &      0.390 \\  
 10 &      380 &       1.194 &      0.168 \\  
 11 &      539 &       1.068 &      0.328 \\  
\hline
\end{tabular}%
\hfill%
\begin{tabular}[t]{r|rcc}
\hline 
$\ell$  &  $\#\TT_\ell$  & $\frac{\#\MM_\ell}{\#\MM_\ell^\eta} $ & $ \frac{\osc_\ell(\MM_\ell^\eta)^2}{\osc_\ell^2} $ \\
\hline 
\hline 
 12 &      721 &       1.050 &      0.346 \\  
 13 &      991 &       1.007 &      0.466 \\  
 14 &     1,356 &       1.003 &      0.482 \\  
 15 &     1,852 &       1.020 &      0.386 \\  
 16 &     2,534 &       1.000 &      0.630 \\  
 17 &     3,413 &       1.009 &      0.443 \\  
 18 &     4,684 &       1.000 &      0.597 \\  
 19 &     6,341 &       1.003 &      0.443 \\  
 20 &     8,568 &       1.002 &      0.490 \\  
 21 &    11,564 &       1.000 &      0.640 \\  
 22 &    15,590 &       1.000 &      0.539 \\  
 23 &    21,071 &       1.000 &      0.569 \\  
\hline
\end{tabular}%
\hfill%
\begin{tabular}[t]{r|rcc}
\hline 
$\ell$  &  $\#\TT_\ell$  & $\frac{\#\MM_\ell}{\#\MM_\ell^\eta} $ & $ \frac{\osc_\ell(\MM_\ell^\eta)^2}{\osc_\ell^2} $ \\
\hline 
\hline 
 24 &    28,304 &       1.017 &      0.437 \\  
 25 &    38,350 &       1.000 &      0.670 \\  
 26 &    51,122 &       1.016 &      0.414 \\  
 27 &    69,135 &       1.000 &      0.563 \\  
 28 &    92,367 &       1.000 &      0.528 \\  
 29 &   123,666 &       1.008 &      0.463 \\  
 30 &   166,532 &       1.000 &      0.703 \\  
 31 &   221,144 &       1.020 &      0.378 \\  
 32 &   298,213 &       1.000 &      0.549 \\  
 33 &   397,086 &       1.000 &      0.597 \\  
 34 &   532,432 &       1.017 &      0.409 \\  
 35 &   712,738 &       1.000 &      0.666 \\
\hline 
\end{tabular}
\end{center}
\caption{Experiment with smooth solution from \cref{ex:bsp2}: We compute $\widetilde{C}_{MNS}:=\#\MM_\ell/\#\MM_\ell^\eta \le 1.8$. 
Hence, the additional assumption in \cref{theorem:mns} is experimentally verified. 
Furthermore, 
we compute $\widetilde\theta':=\osc_\ell(\MM_\ell^\eta)^2/\osc_\ell^2 \ge 0.02$
with $\osc_{\ell}(\MM_\ell^\eta)^2:=\sum_{T\in\MM_\ell^\eta} \osc_\ell(T,u_\ell)^2$,} i.e., 
the choice $\theta = 0.5$, $\theta' = 0.02$ would guarantee $\MM_\ell = \MM_\ell^\eta$ in \cref{algorithm:mns}.
\label{tab:bsp2}
\end{table}

\subsection{Experiment with generic singularity}
\label{ex:bsp2}
On the L-shaped domain $\Omega=\linebreak(-1,1)^2\backslash \big([0,1]\times[-1,0]\big)$
we consider the exact solution 
$u(x_1,x_2) = r^{2/3}\sin(2\varphi/3)$ 
in polar coordinates $r\in\R_0^+$, $\varphi\in[0,2\pi[$,
and $(x_1,x_2) = r(\cos\varphi,\sin\varphi)$.
It is well known that $u$ has a generic singularity at the reentrant corner $(0,0)$, which
leads to $u\in H^{1+2/3-\varepsilon}(\Omega)$ for all
$\varepsilon>0$.
We choose the diffusion matrix 
\begin{align*}
	\A=  \left ( \begin{array}{rr}
  5+(x_1^2+x_2^2)\cos x_1 & (x_1^2+x_2^2)^2 \\
  (x_1^2+x_2^2)^2 & \;5+(x_1^2+x_2^2)\sin x_2
  \end{array}\right)
\end{align*}
so that~\cref{eq:A} holds with $\lambda_{\min}=0.46689$ and
$\lambda_{\max}=5.14751$, and $\b=(1,1)^T$ and $\c=1$ so that \cref{eq:bcestimate1}
holds with $\frac{1}{2}\div\,\b+\c=1$.
The right-hand side $f$ is calculated appropriately.
The uniform initial mesh $\TT_0$ consists of $12$ triangles. An adaptively generated 
mesh after $16$ refinements and 
a plot of the discrete solution are shown in \cref{fig:bsp2meshsol}.

We observe the expected suboptimal convergence order of $\OO(N^{-1/3})$ for uniform mesh-refinement. 
We regain the optimal convergence order of $\OO(N^{-1/2})$ for adaptive mesh-refinement;
see \cref{fig:bsp2error}. As in the experiment of \cref{ex:bsp1}, 
the oscillations are of higher order $\OO(N^{-1})$.
We refer to \cref{tab:bsp2} for the experimental validation of the additional assumption in 
\cref{theorem:mns}~that marking for the data oscillations is negligible.


\begin{figure}
\begin{center}
	\subfigure[\label{subfig:bsp3solT8uni}Solution $\TT_{8}$ (uniform, 8192 elements).]
	{\includegraphics[width=0.46\textwidth]
	{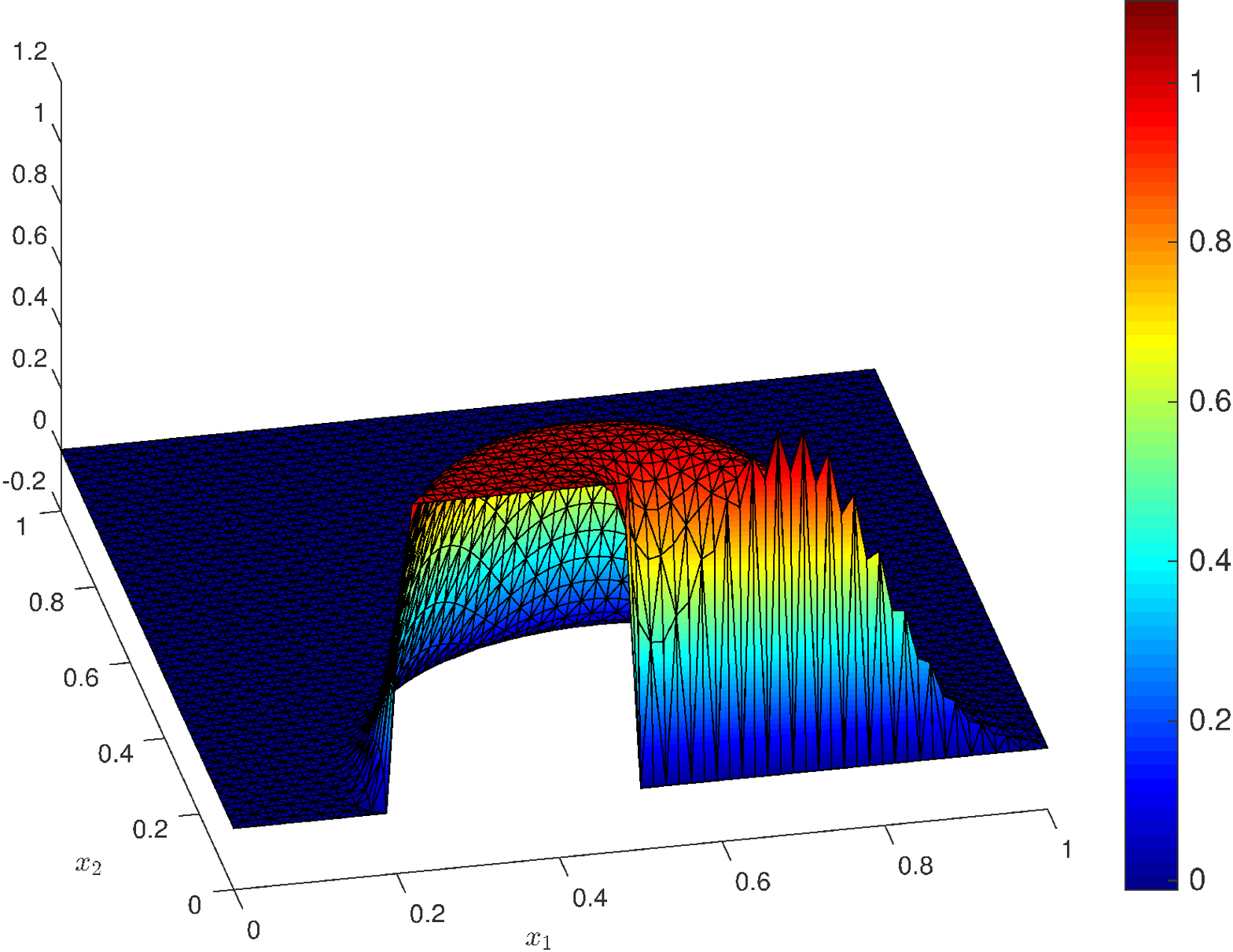}}
	\hspace{0.05\textwidth}
	\subfigure[\label{subfig:bsp3solT14}Solution $\TT_{14}$ (adaptive, 779 elements).]
	{\includegraphics[width=0.46\textwidth]
	{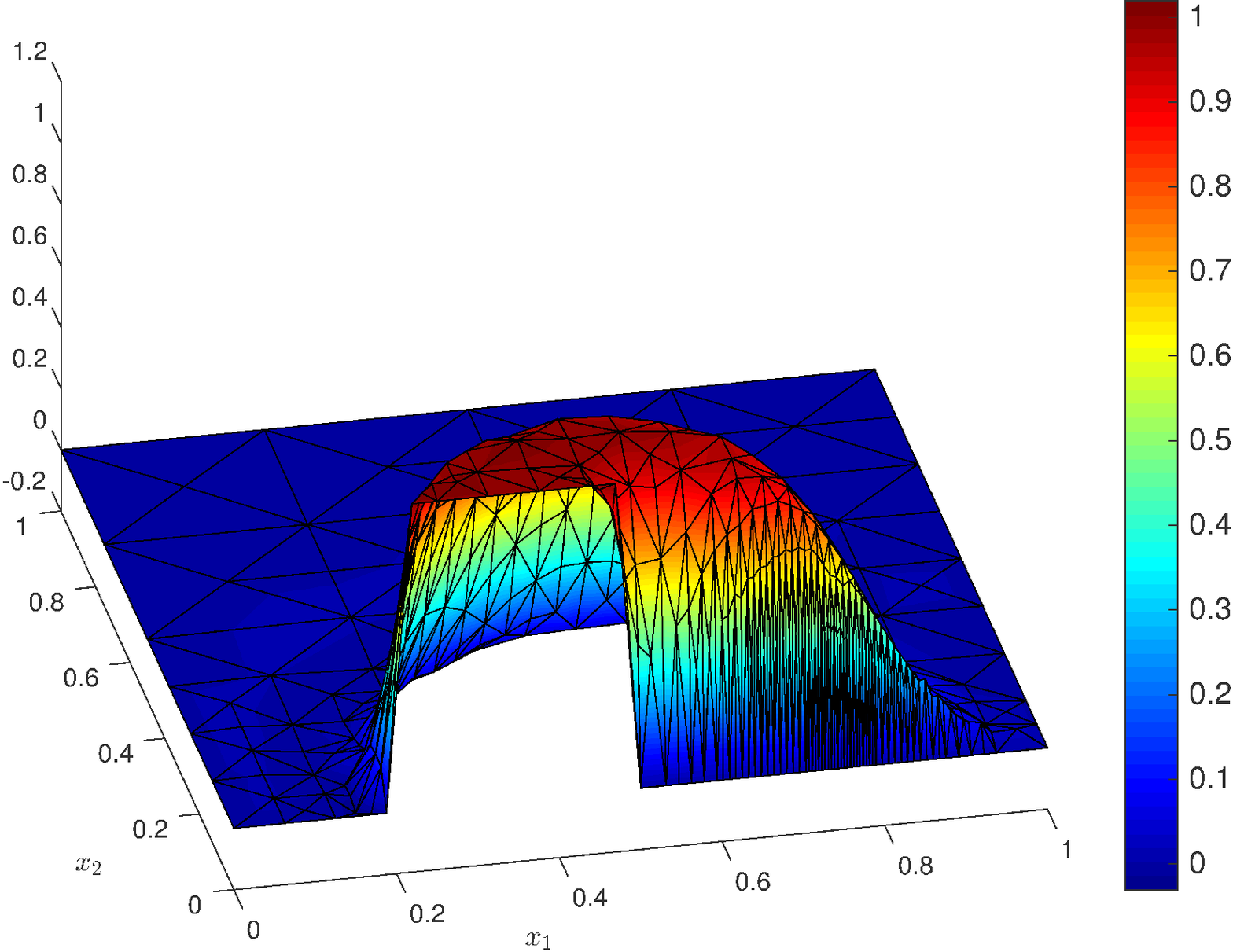}}
\end{center}
\caption{\label{fig:bsp3sol}%
Convection dominated experiment from \cref{ex:bsp3}:
The discrete FVM solution on an uniformly generated mesh $\TT_{8}$
and adaptively generated mesh $\TT_{14}$. 
The algorithm starts with a uniform initial triangulation $\TT_0$ with $32$ elements.}
\end{figure}

\begin{figure}
\begin{center}
	\subfigure[\label{subfig:bsp3meshT14}$\TT_{14}$ (779 elements).]
	{\includegraphics[width=0.3\textwidth]
	{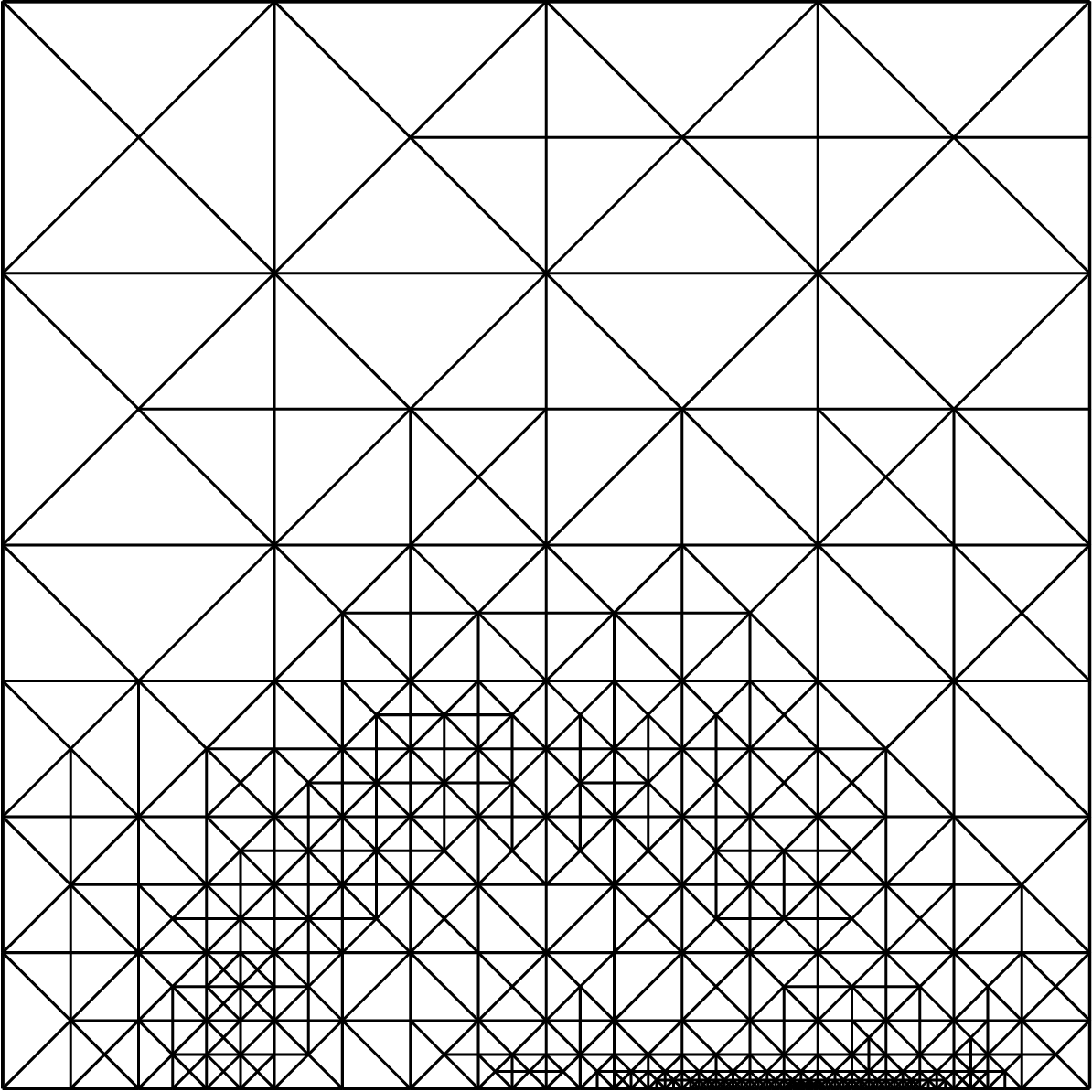}}
	\hspace{0.25\textwidth}
	\subfigure[\label{subfig:bsp3meshT20}$\TT_{20}$ (4336 elements).]
	{\includegraphics[width=0.3\textwidth]
	{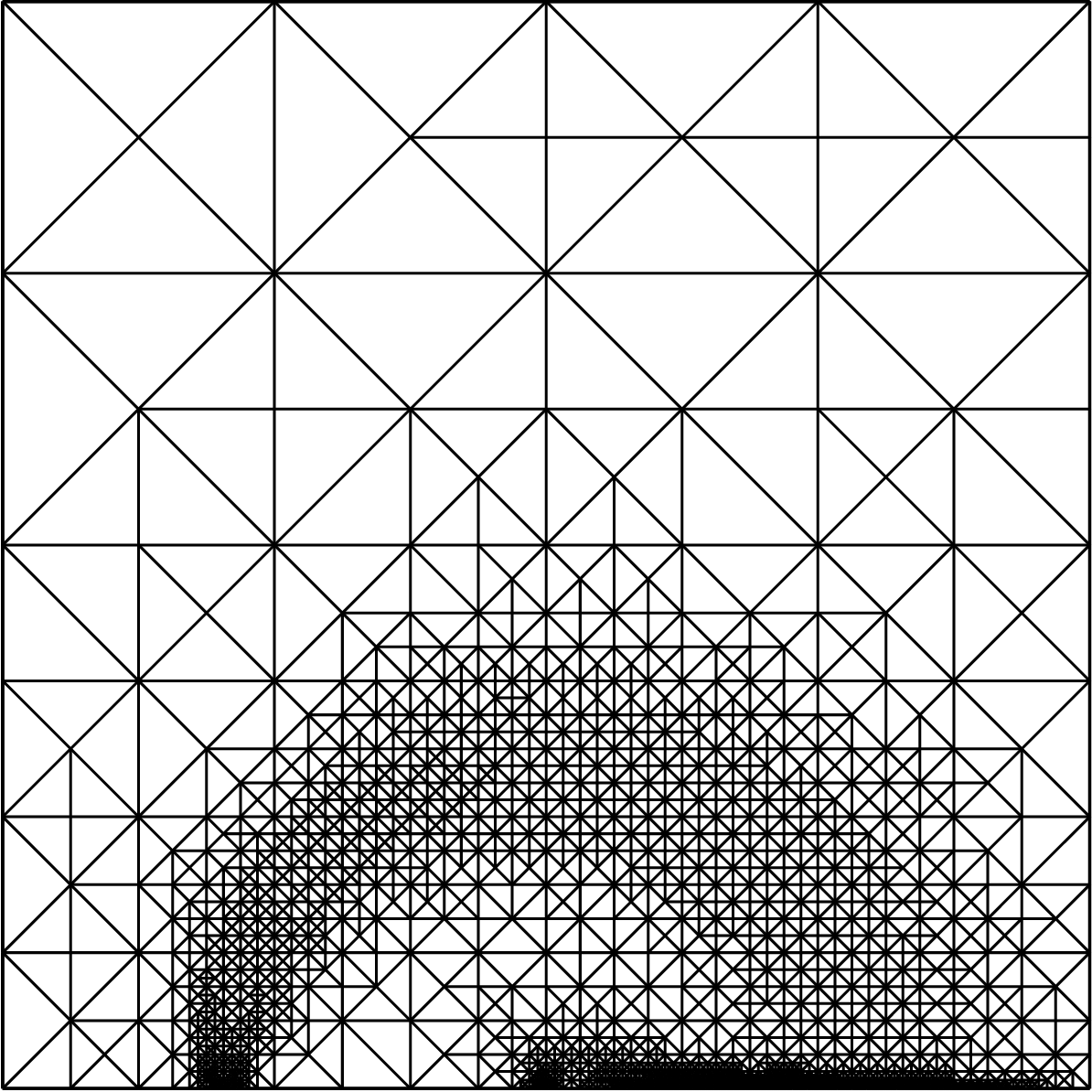}}
\end{center}
\caption{\label{fig:bsp3mesh}%
Convection dominated experiment from \cref{ex:bsp3}:
Adaptively generated meshes $\TT_{14}$ and $\TT_{20}$ from a
uniform initial triangulation $\TT_0$ with $32$ elements.}
\end{figure}

\begin{figure}
\begin{center}
\begin{psfrags}%
\psfragscanon%
%
\psfrag{s05}[l][l]{\small $\osc_\ell$ (uni.)}%
\psfrag{s07}[t][t]{\small number of elements}%
\psfrag{s01}[l][l]{\small $\eta_\ell$ (uni.)}%
\psfrag{s06}[b][b]{\small estimator, oscillation}%
\psfrag{s08}[l][l]{\small $\osc_\ell$ (ada.)}%
\psfrag{s04}[l][l]{\scriptsize $\;\;\;1/3$}%
\psfrag{s12}[l][l]{\small $\eta_\ell$ (ada.)}%
\psfrag{s04}[l][l]{\scriptsize $\;\;1/2$}%
\psfrag{s14}[l][l]{\scriptsize $\;\;1/2$}%
\psfrag{s03}[l][l]{\scriptsize $1$}%
\psfrag{s11}[l][l]{\scriptsize $1$}%
\psfrag{s09}[l][l]{\scriptsize $1$}%
\psfrag{s13}[l][l]{\scriptsize $\;\;\;\;1$}%

%
\color[rgb]{0.15,0.15,0.15}%
%
\psfrag{x01}[t][t]{\scriptsize${10^{1}}$}%
\psfrag{x02}[t][t]{\scriptsize${10^{2}}$}%
\psfrag{x03}[t][t]{\scriptsize${10^{3}}$}%
\psfrag{x04}[t][t]{\scriptsize${10^{4}}$}%
\psfrag{x05}[t][t]{\scriptsize${10^{5}}$}%
\psfrag{x06}[t][t]{\scriptsize${10^{6}}$}%
%
\psfrag{v01}[r][r]{\scriptsize${10^{-6}}$}%
\psfrag{v02}[r][r]{\scriptsize${10^{-5}}$}%
\psfrag{v03}[r][r]{\scriptsize${10^{-4}}$}%
\psfrag{v04}[r][r]{\scriptsize${10^{-3}}$}%
\psfrag{v05}[r][r]{\scriptsize${10^{-2}}$}%
\psfrag{v06}[r][r]{\scriptsize${10^{-1}}$}%
\psfrag{v07}[r][r]{\scriptsize${10^{0}}$}%
%
\includegraphics[width=0.8\textwidth]{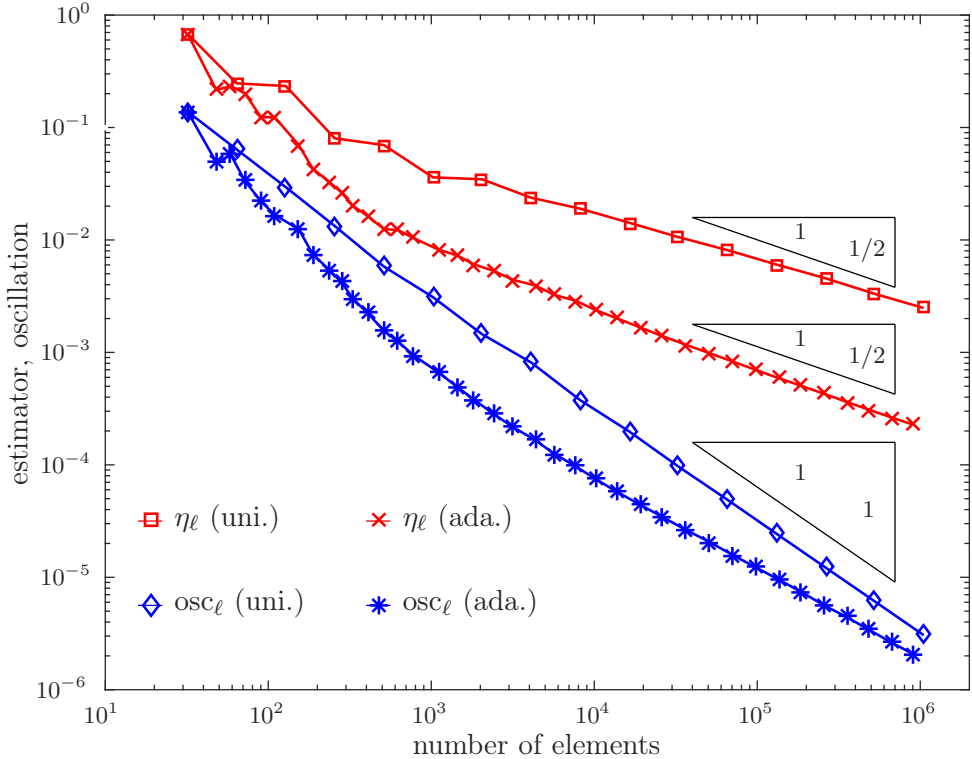}%
\end{psfrags}%
\end{center}
\caption{\label{fig:bsp3error}%
Convection dominated experiment from \cref{ex:bsp3}: 
Weighted-residual error estimator
$\eta_\ell$ and data oscillations $\osc_\ell$ for uniform and adaptive
mesh-refinement.}
\end{figure}
\subsection{Convection dominated experiment}
\label{ex:bsp3}
The final example is taken from~\cite{Mekchay:2005-1}. On the square $\Omega=(0,1)^2$, we
fix the diffusion $\A=10^{-3}\mathbf{I}$ and the convection velocity $\b=(x_2,1/2-x_1)^T$. 
The reaction and right-hand side are $\c=f=0$. Thus,~\cref{eq:A} holds 
with $\lambda_{\min}=\lambda_{\max}=10^{-3}$ and~\cref{eq:bcestimate1} with $\frac{1}{2}\div\,\b+\c=0$.
On the Dirichlet boundary $\Gamma$, we prescribe the continuous piecewise linear function by
\begin{align*}
	u(x_1,x_2)|_\Gamma=  \left \{ \begin{array}{rl}
  1 & \quad\text{on } \{0.2005\leq x_1\leq 0.4995, x_2=0\}, \\
  0 & \quad\text{on } \Gamma\backslash\{0.2\leq x_1\leq 0.5; x_2=0\}\\
  \text{linear} & \quad\text{on } \{0.2\leq x_1\leq 0.2005 \text{ or }0.4995\leq x_1\leq 0.5; x_2=0\},
  \end{array}\right.
\end{align*}
The model has a moderate convection dominance with respect to the diffusion and 
simulates the transport of a pulse from $\Gamma$ to the interior and back to $\Gamma$.
For this example, we do not know the analytical solution. 
The uniform initial mesh $\TT_0$ consists of $32$ triangles. 
In \cref{fig:bsp3sol}\subref{subfig:bsp3solT8uni}, we see the solution with strong oscillations 
on an uniformly generated mesh with
$8192$ elements. The oscillations are due to the convection dominance.
For the next refinement step ($16384$ elements, not plotted), however, the oscillations disappear since the shock region
at the boundary is refined enough.
Our adaptive \cref{algorithm:mns}, which \emph{also} has a mandatory oscillation marking, provides a stable solution 
on a mesh with only $779$ elements; see \cref{fig:bsp3sol}\subref{subfig:bsp3solT14}. 
In \cref{fig:bsp3mesh}, we plot adaptively generated 
meshes after $14$ and $20$ mesh-refinements. We see a strong refinement in the shock region.
A similar observation can be found in~\cite{Mekchay:2005-1}.
We remark that this strategy only works for this moderate convection dominated problem. For $\A=10^{-8}\mathbf{I}$,
we cannot see any stabilization effects by \cref{algorithm:mns} (not displayed).
Hence, only a stabilization of the numerical scheme, e.g., FVM with upwinding, would avoid these instabilities. 
However, the analysis of 
such schemes is beyond the scope of this work.
We observe the above stabilization effects also in the convergence plot of the estimator; see \cref{fig:bsp3error}.
Note that the estimator for adaptive mesh-refinement is faster in the asymptotic convergence than 
the estimator for uniform
mesh-refinement.
Additionally, the convergence rate for the estimator is suboptimal for uniform mesh-refinement. 
For adaptive mesh-refinement, we regain the 
optimal convergence order of $\OO(N^{-1/2})$;
see \cref{fig:bsp3error}. As in the previous experiments, the oscillations are of higher order.
In \cref{tab:bsp3}, we also see that the oscillation
marking for this convection dominated problem is for more refinement steps dominant 
than for the previous problems; see also 
the discussion in~\cite{Mekchay:2005-1}.
\begin{table}[!t]\small
\begin{center}
\begin{tabular}[t]{r|rcc}
\hline 
$\ell$  &  $\#\TT_\ell$  & $\frac{\#\MM_\ell}{\#\MM_\ell^\eta} $ & $ \frac{\osc_\ell(\MM_\ell^\eta)^2}{\osc_\ell^2} $ \\
\hline 
\hline 
  0 &       32 &  1.125 &      0.434 \\  
  1 &       48 &  1.400 &      0.201 \\  
  2 &       59 &  1.500 &      0.266 \\  
  3 &       72 &  1.667 &      0.196 \\  
  4 &       90 &  2.500 &      0.177 \\  
  5 &      110 &  1.333 &      0.266 \\  
  6 &      154 &  1.583 &      0.085 \\  
  7 &      187 &  1.500 &      0.124 \\  
  8 &      238 &  1.786 &      0.055 \\  
  9 &      280 &  1.296 &      0.234 \\  
 10 &      332 &  1.371 &      0.154 \\  
 11 &      405 &  1.412 &      0.124 \\  
 12 &      511 &  1.537 &      0.083 \\  
\hline 
\end{tabular}%
\hfill%
\begin{tabular}[t]{r|rcc}
\hline 
$\ell$  &  $\#\TT_\ell$ & $\frac{\#\MM_\ell}{\#\MM_\ell^\eta} $ & $ \frac{\osc_\ell(\MM_\ell^\eta)^2}{\osc_\ell^2} $ \\
\hline 
\hline 
 13 &      628 &  1.521 &      0.146 \\  
 14 &      779 &  1.559 &      0.077 \\  
 15 &     1,100 &  1.600 &      0.064 \\  
 16 &     1,428 &  1.605 &      0.063 \\  
 17 &     1,837 &  1.643 &      0.037 \\  
 18 &     2,416 &  1.594 &      0.058 \\  
 19 &     3,195 &  1.437 &      0.060 \\   
 20 &     4,336 &  1.583 &      0.048 \\  
 21 &     5,664 &  1.402 &      0.072 \\  
 22 &     7,666 &  1.445 &      0.047 \\  
 23 &    10,186 &  1.351 &      0.067 \\  
 24 &    13,919 &  1.258 &      0.078 \\  
 25 &    19,041 &  1.230 &      0.112 \\  
\hline 
\end{tabular}%
\hfill%
\begin{tabular}[t]{r|rcc}
\hline 
$\ell$  &  $\#\TT_\ell$ & $\frac{\#\MM_\ell}{\#\MM_\ell^\eta} $ & $ \frac{\osc_\ell(\MM_\ell^\eta)^2}{\osc_\ell^2} $ \\
\hline 
\hline 
 26 &    26,248 &  1.182 &      0.106 \\  
 27 &    36,592 &  1.142 &      0.135 \\  
 28 &    50,806 &  1.112 &      0.180 \\  
 29 &    70,367 &  1.082 &      0.196 \\  
 30 &    97,946 &  1.058 &      0.227 \\  
 31 &   135,122 &  1.057 &      0.236 \\  
 32 &   186,959 &  1.028 &      0.311 \\  
 33 &   255,994 &  1.021 &      0.311 \\  
 34 &   351,880 &  1.022 &      0.289 \\  
 35 &   484,157 &  1.015 &      0.328 \\  
 36 &   662,325 &  1.006 &      0.381 \\  
 37 &   902,659 &  1.005 &      0.384 \\ 
\hline 
\end{tabular}
\end{center}
\caption{Experimental results on marking strategy
for the convection dominated experiment from \cref{ex:bsp3}:
We compute $\widetilde{C}_{MNS}:=\#\MM_\ell/\#\MM_\ell^\eta \le 3$ and see that the additional assumption in \cref{theorem:mns} is experimentally verified. In addition, we compute
$\widetilde\theta':=\osc_\ell(\MM_\ell^\eta)^2/\osc_\ell^2 \ge 0.03$
with $\osc_{\ell}(\MM_\ell^\eta)^2:=\sum_{T\in\MM_\ell^\eta} \osc_\ell(T,u_\ell)^2$,
i.e., the choice $\theta = 0.5$, $\theta' = 0.03$ would guarantee $\MM_\ell = \MM_\ell^\eta$ in \cref{algorithm:mns}.}
\label{tab:bsp3}
\end{table}
\section{Conclusions}
%
In this work, we have proved linear convergence of an adaptive vertex-centered
finite volume method with generically optimal algebraic rates to the solution of a general second-order
linear elliptic PDE. Besides marking based on the local contributions of the {\sl a~posteriori}
error estimator, we additionally had to mark the 
oscillations
to overcome the lack of a classical Galerkin orthogonality property.
In case of dominating convection, finite volume methods provide a natural upwind stabilization.
Although there exist estimators also for these upwind discretizations~\cite{Erath:2013-1}, 
we were not able to provide a 
rigorous convergence result for the related adaptive mesh-refinement strategy.
Note that the upwind direction and thus the corresponding error indicator contributions
are defined over the boundary of the control volumes of the dual mesh.
As mentioned above, the dual meshes are not nested even for a sequence of locally refined triangulations. 
This makes it difficult to show~\cref{eq:A1}--\cref{eq:A2} and~\cref{eq:B1}--\cref{eq:B2}. 
We stress that the other error indicator contributions are defined over the elements of the primal mesh and can hence be treated by the developed techniques.

\bibliographystyle{alpha}
\bibliography{afvmgeneral}  

\end{document}